\pdfoutput=1
\documentclass[11pt]{amsart}
\usepackage[letterpaper,margin=1in]{geometry}
\usepackage[T1]{fontenc}
\usepackage{lmodern,amsmath,amssymb,amsthm,mathtools,enumitem}
\usepackage[expansion=false]{microtype}
\usepackage[hidelinks]{hyperref}
\allowdisplaybreaks[2]
\numberwithin{equation}{section}
\newtheorem{theorem}{Theorem}
\newtheorem{proposition}{Proposition}[section]
\newtheorem{lemma}[proposition]{Lemma}
\newtheorem{corollary}[proposition]{Corollary}
\theoremstyle{definition}
\newtheorem{definition}[proposition]{Definition}
\newtheorem{example}[proposition]{Example}
\theoremstyle{remark}
\newcommand{\R}{\mathbb R}
\newcommand{\Z}{\mathbb Z}
\newcommand{\HH}{\mathbb H}
\newcommand{\Pc}{\Pi_{\mathrm{ac}}}
\newcommand{\Pczer}{\Pi_{\mathrm{ac},0}}
\newcommand{\supp}{\operatorname{supp}}
\newcommand{\dist}{\operatorname{dist}}
\newcommand{\inj}{\operatorname{inj}}
\newcommand{\Id}{\operatorname{Id}}
\newcommand{\1}{\mathbf1}
\newcommand{\dd}{\,d}
\newcommand{\norm}[2]{\left\|#1\right\|_{#2}}
\newcommand{\ip}[2]{\left\langle#1,#2\right\rangle}
\newcommand{\mug}{\mu_\gamma}
\title[Spectral and Schr\"odinger restriction estimates]{Spectral and Schr\"odinger estimates restricted to submanifolds}
\author{Zhexing Zhang}
\date{}
\hypersetup{pdftitle={Spectral and Schr\"odinger estimates restricted to submanifolds},pdfauthor={Zhexing Zhang}}
\begin{document}
\begin{abstract}
Let $M$ be a complete Riemannian manifold with nonpositive sectional curvature and smooth bounded geometry. We prove logarithmic spectral projection estimates restricted to uniformly embedded submanifolds. On surfaces with pinched negative curvature, we also prove frequency-independent $L^2\to L^3$ Schr\"odinger restriction estimates on compact geodesic segments for times $c\lambda^{-1}\log\lambda$. If such a surface is even asymptotically hyperbolic, the $L^2\to L^q$ spectral estimates on complete nontrapped geodesics, for $q>2$, gain the square root of the window width. The high-frequency Schr\"odinger estimate holds for all times on these surfaces, both on compact segments and on complete nontrapped geodesics. Under a regularity assumption at the spectral threshold, the all-time estimate holds on the absolutely continuous spectral subspace.
\end{abstract}
\maketitle
\section{Introduction}
Let $(M,g)$ be a smooth connected complete Riemannian manifold of dimension $n\ge2$, without boundary, with bounded geometry and nonpositive sectional curvature. Let $\Sigma$ be a $k$-dimensional uniformly embedded smooth submanifold, with $0\le k\le n-1$. Throughout, $\Delta_g=\operatorname{div}_g\nabla\le0$, $P=\sqrt{-\Delta_g}$ and $U(t)=e^{it\Delta_g}$. Denote by $R_\Sigma$ the restriction to $\Sigma$ and by $\mathbf1_I(P)$ the spectral projection onto $I$. Set

\begin{equation}\label{eq:mu}
\mu(q)=\begin{cases}
\dfrac{n-1}{2}-\dfrac{k}{q},&k\le n-2,\ q\ge2,\ \text{or }k=n-1,\ q\ge\dfrac{2n}{n-1},\\[1mm]
\dfrac{n-1}{4}-\dfrac{k-1}{2q},&k=n-1,\ 2\le q<\dfrac{2n}{n-1}.
\end{cases}
\end{equation}

On compact manifolds, Burq, G\'erard and Tzvetkov~\cite{BGT-restriction} proved the unit-band restriction bounds

\begin{equation}\label{eq:BGT}
\norm{R_\Sigma\1_{[\lambda,\lambda+1]}(P)f}{L^q(\Sigma)}
\lesssim\lambda^{\mu(q)}\norm f2
\end{equation}

except at $(k,q)=(n-2,2)$ and, for $n\ge3$, at the hypersurface exponent $q=2n/(n-1)$, where their general estimate includes a factor $(\log\lambda)^{1/2}$:

\begin{equation}\label{eq:BGT-log}
\norm{R_\Sigma\1_{[\lambda,\lambda+1]}(P)f}{L^q(\Sigma)}
\lesssim\lambda^{\mu(q)}(\log\lambda)^{1/2}\norm f2
\end{equation}

For curves on surfaces, the endpoint $q=4$ has no logarithmic loss by~\cite[Theorem~1]{BGT-restriction}. Hu~\cite{Hu} removed the hypersurface endpoint logarithm in higher dimensions. On compact manifolds with nonpositive sectional curvature, Chen~\cite{Chen} used logarithmic windows to improve the strict non-endpoint ranges. Our first result extends Chen's estimate to complete manifolds of smooth bounded geometry.

\begin{theorem}\label{thm:static-general}
Suppose that $\Sigma$ is uniformly embedded in the smooth sense of Definition~\ref{def:embedding}. If $k=n-1$ and $q>2n/(n-1)$, or if $k\le n-2$ and $q>2$, then
\begin{equation}\label{eq:main-static}
\norm{R_\Sigma\1_{[\lambda,\lambda+(\log\lambda)^{-1}]}(P)f}{L^q(\Sigma)}
\lesssim \lambda^{\mu(q)}(\log\lambda)^{-1/2}\norm f{L^2(M)},\qquad\lambda\ge2.
\end{equation}
The constants are allowed to depend on the fixed geometry, the uniform embedding bounds, and $q$.
\end{theorem}

We next specialize to geodesic restriction on surfaces with strictly negative curvature. Restrictions of eigenfunctions to curves on compact hyperbolic surfaces were investigated by Reznikov~\cite{Reznikov}. Xi and Zhang~\cite{Xi,XiZhang} studied logarithmic spectral projections on compact geodesics. Chen and Sogge~\cite{ChenSogge} obtained the improvement $o(\lambda^{1/4})$ at the critical $L^4$ restriction exponent. Blair~\cite{Blair} proved

\begin{equation}\label{eq:Blair-intro}
\norm{R_\gamma\1_{[\lambda,\lambda+(\log\lambda)^{-1}]}(P)}{L^2(M)\to L^4(\gamma)}
\lesssim\lambda^{1/4}(\log\lambda)^{-1/4}
\end{equation}

for compact surfaces of nonpositive curvature. For compact negatively curved surfaces, the gain $(\log\lambda)^{-1/2}$ follows from~\cite[Remark~1]{CZhang}; see also~\cite[(2.10)]{HWZ}. Proposition~\ref{prop:critical-log} extends this estimate to compact geodesic segments on complete surfaces of smooth bounded geometry with pinched negative curvature. We use this estimate in the spectral-window gluing below.

Let $(M,g)$ now be a smooth even asymptotically hyperbolic surface. Thus $M$ is the interior of a compact smooth manifold $\overline M$ with nonempty boundary, and near the boundary its metric is

\begin{equation}\label{eq:AH-definition}
 g=\frac{dx_1^2+g_1(x_1^2)}{x_1^2},
\end{equation}

where $x_1$ is a boundary defining function and $g_1$ is a smooth family of metrics on $\partial\overline M$. A complete unit-speed geodesic $\gamma:\R\to M$ is \emph{nontrapped} when $\gamma(r)$ leaves every compact set as $r\to\pm\infty$. Huang, Sogge, Tao and the author~\cite{HSTZ} proved lossless ambient spectral-window estimates on negatively curved asymptotically hyperbolic surfaces. We prove the corresponding restriction estimate on complete nontrapped geodesics.

\begin{theorem}\label{thm:AH-static}
Let $(M,g)$ be a smooth connected complete even asymptotically hyperbolic surface with smooth bounded geometry and $-b^2\le K_g\le-a^2<0$. Let $\gamma$ be a complete nontrapped geodesic. For every fixed $2<q\le\infty$, $\lambda\ge1$ and $0<\eta\le1$,
\begin{equation}\label{eq:main-AH}
 \norm{R_\gamma\1_{[\lambda,\lambda+\eta]}(P)f}{L^q(\gamma)}
 \lesssim_{q,\gamma}\lambda^{\mug(q)}\eta^{1/2}\norm f{L^2(M)},
\end{equation}
where
\begin{equation}\label{eq:mu-surface}
 \mug(q)=\begin{cases}1/4,&2\le q\le4,\\1/2-1/q,&4\le q\le\infty.\end{cases}
\end{equation}
The norm is taken with respect to the arclength parameter of $\gamma$.
\end{theorem}

Examples include noncompact convex cocompact hyperbolic surfaces with funnels and no cusps, as discussed in~\cite{Borthwick}. Anker, Germain and L\'eger~\cite{AGL} established lossless small-window spectral estimates under a pressure condition. Their hyperbolic-plane formulas also motivate the sharpness examples below. Improvements in arithmetic settings were obtained by Marshall~\cite{Marshall} for Hecke--Maass forms on compact congruence surfaces and by Hou~\cite{Hou} for totally geodesic surfaces in arithmetic hyperbolic three-manifolds. On the standard flat tori of dimensions $2$ and $3$, Bourgain and Rudnick~\cite{BR} proved $L^2$ restriction bounds for hypersurfaces of nonvanishing curvature.

We next consider spacetime restriction for the Schr\"odinger evolution. For a unit-speed segment $\gamma:I\to M$, restriction is taken in the arclength parameter, including multiplicity at self-intersections. Fix $\beta\in C_c^\infty((1/2,2))$. The two-dimensional $L^2$-preserving Euclidean scaling $u_\lambda(t,x)=\lambda u(\lambda^2t,\lambda x)$ yields

\begin{equation}\label{eq:scaling}
 \norm{u_\lambda|_{x_2=0}}{L^p(\R_t\times\R_{x_1})}
 =\lambda^{1-3/p}\norm{u|_{x_2=0}}{L^p(\R_t\times\R_{x_1})}.
\end{equation}

Thus $p=3$ is the scale-invariant spacetime exponent. Tacy's local estimate~\cite{Tacy} supplies

\begin{equation}\label{eq:intro-local}
 \norm{R_\gamma U(t)\beta(P/\lambda)f}{L^3([0,\lambda^{-1}]\times\gamma)}
 \lesssim\norm f2
\end{equation}

on a fixed short geodesic segment, uniformly under bounded geometry. Subdivision gives a factor $\lambda^{1/3}$ on a unit time interval. We extend the local bound to logarithmic semiclassical times.

\begin{theorem}\label{thm:log}
Let $(M,g)$ be a connected complete surface of smooth bounded geometry whose Gaussian curvature satisfies
\begin{equation}\label{eq:pinching}
 -b^2\le K_g\le-a^2<0.
\end{equation}
For a fixed $\beta\in C_c^\infty((1/2,2))$, there exist $c,C>0$ and $\lambda_0\ge2$ such that
\begin{equation}\label{eq:main-log}
 \norm{R_\gamma U(t)\beta(P/\lambda)f}
 {L^3([0,c\lambda^{-1}\log\lambda]\times\gamma)}
 \le C\norm f{L^2(M)},\qquad \lambda\ge\lambda_0,
\end{equation}
for every unit-speed geodesic segment with $|I|\le1$. The constants are uniform over these segments on the fixed surface. The same conclusion holds for each fixed compact segment, with a constant depending on its length.
\end{theorem}

\begin{corollary}\label{cor:closed}
If $M$ is a fixed closed negatively curved surface and $\gamma:I\to M$ is a fixed compact geodesic segment, then
\begin{equation}\label{eq:closed-unit}
 \norm{R_\gamma U(t)\beta(P/\lambda)f}{L^3([0,1]\times\gamma)}
 \lesssim \lambda^{1/3}(\log\lambda)^{-1/3}\norm f{L^2(M)}.
\end{equation}
\end{corollary}

The corollary follows by time subdivision. The high-frequency order in the logarithmic-time estimate is sharp whenever $\beta\not\equiv0$.

To prove the geodesic estimates, we lift the wave kernel to the universal cover and divide the deck transformations into tube and oscillatory terms. The tube terms are estimated by orbit counting and Jacobi comparison. For the other terms, we use Blair's finite-type estimate. In the Schr\"odinger argument, we first take the Fourier transform in the time difference to obtain an oscillatory integral with the distance phase.

At infinity we use the scalar spectral-measure bounds of Chen and Hassell~\cite{CH} on simply connected auxiliary surfaces. Integrating the density gives the spectral-window bound. For Schr\"odinger evolution, temporal Plancherel gives the $L^2$ endpoint of the spacetime $TT^*$ estimate. The auxiliary surface construction and local smoothing estimates of~\cite{HSTZ} transfer both bounds to the original surface. Retarded spacetime integrals are estimated by Christ--Kiselev~\cite{CK}.

Write $\Pc$ for the orthogonal projection onto the absolutely continuous spectral subspace of $-\Delta_g$. On an auxiliary background $(Y,g_0)$ the corresponding projection is $\Pczer$. For an asymptotically hyperbolic surface this spectrum is $[1/4,\infty)$, with finitely many eigenvalues below the threshold, as in~\cite{CH}. We call the threshold \emph{regular} if $1/4$ is not a resonance of the continued resolvent. The all-time result is as follows.

\begin{theorem}\label{thm:asymp}
Let $(M,g)$ be a smooth connected complete even asymptotically hyperbolic surface satisfying~\eqref{eq:pinching}. Let $\gamma:I\to M$ be either a fixed compact geodesic segment or a complete nontrapped geodesic, in which case $I=\R$. There exist $\lambda_0>1$ and $C_\gamma$ such that
\begin{equation}\label{eq:asymp-high}
 \norm{R_\gamma U(t)\1_{[\lambda_0,\infty)}(P)f}{L^3(\R\times\gamma)}
 \le C_\gamma\norm f{L^2(M)}.
\end{equation}
If the threshold $1/4$ is regular, then
\begin{equation}\label{eq:asymp-full}
 \norm{R_\gamma U(t)\Pc f}{L^3(\R\times\gamma)}
 \le C_\gamma\norm f{L^2(M)}.
\end{equation}
\end{theorem}

An eigenfunction with a nonzero trace has infinite $L^3(\R\times\gamma)$ norm under Schr\"odinger evolution, which accounts for the projection $\Pc$. The surface itself may have a trapped set.

Local semiclassical Strichartz estimates on compact manifolds were obtained by Burq, G\'erard and Tzvetkov~\cite{BGT-Strichartz}. Huang and Sogge~\cite{HS} established lossless logarithmic-time bounds under nonpositive curvature. Chen~\cite{ChenIII} used Chen--Hassell spectral measures to prove all-time Strichartz estimates on nontrapping asymptotically hyperbolic manifolds under threshold regularity. Time localization and smoothing in the presence of trapping appear in~\cite{BGH,HSTZ}, and Nicola~\cite{Nicola} proved a transference principle from Fourier restriction estimates to Schr\"odinger propagators.

Section~\ref{sec:static-general} proves the logarithmic spectral estimates, including the critical geodesic bound. Section~\ref{sec:AH-static} proves Theorem~\ref{thm:AH-static}. Sections~\ref{sec:logproof} and~\ref{sec:gluing} prove the logarithmic-time and all-time Schr\"odinger estimates. Section~\ref{sec:sharp} gives sharpness examples and the threshold obstruction.

\subsection*{Notation}
For nonnegative $A$ and $B$, we write $A\lesssim B$ if $A\le CB$ with $C$ independent of the frequency and time parameters. Constants may depend on the fixed geometry, the observation set, the exponents and the cutoffs. For uniformly embedded submanifolds, the uniform embedding bounds are included in this dependence. We write $A\asymp B$ if both inequalities hold. We introduce $h=\lambda^{-1}$ in the semiclassical sections. We use the graph Sobolev norm $\|v\|_{H^m(M)}=\|(1-\Delta_g)^{m/2}v\|_2$. For a self-adjoint operator $A$, let $E_A$ denote its spectral resolution. Spatial cutoffs are real-valued and $[A,B]=AB-BA$. We sometimes use the same notation for a kernel and its integral operator.
We use the temporal Fourier transform
\begin{equation}\label{eq:FT-convention}
 \widehat F(\tau)=\int_\R e^{it\tau}F(t)\dd t,
 \qquad
 F(t)=\frac1{2\pi}\int_\R e^{-it\tau}\widehat F(\tau)\dd\tau.
\end{equation}

\subsection*{Acknowledgement}
The author would like to thank Daniel Pezzi, Connor Quinn and Christopher Sogge for their helpful comments and advice.

\section{Logarithmic spectral windows}\label{sec:static-general}
\subsection{Uniformly embedded submanifolds}
We use the smooth bounded-geometry formulation in Chapter~2 of~\cite{Eldering}.
\begin{definition}[Manifold with bounded geometry]\label{def:bg}
A complete Riemannian manifold $(M,g)$ has bounded geometry if its injectivity radius is positive and the curvature tensor and all of its covariant derivatives are uniformly bounded.
\end{definition}
There is a $\delta(M)>0$ such that the coordinate charts given by $\exp_x^M$ are uniformly defined on balls of radius $\delta(M)$. In these normal coordinates the metric, its inverse and their derivatives of each fixed order are uniformly bounded. Distances are uniformly equivalent to Euclidean distances on smaller balls, and the transition maps have uniform derivative bounds.

\begin{definition}[Uniformly embedded submanifold]\label{def:embedding}
Let $\iota:\Sigma\to M$ be a smooth embedding. Denote by $\Sigma(x,r)$ the component through $\iota(x)$ of $B_M(\iota(x),r)\cap\iota(\Sigma)$. We say that $\Sigma$ is uniformly embedded if there is $\varpi(\Sigma)>0$ such that, for every $x\in\Sigma$:
\begin{enumerate}[label=(\arabic*),leftmargin=2em]
\item In normal coordinates at $\iota(x)$, $\Sigma(x,\varpi(\Sigma))$ is the graph of a function $h_x:T_x\Sigma\to N_x\Sigma$ over a fixed-radius tangent ball. After orthonormal identifications, the functions $h_x$ and their derivatives of each fixed order have bounds independent of $x$.
\item This is the unique component of $\Sigma\cap B_M(\iota(x),\varpi(\Sigma))$.
\end{enumerate}
\end{definition}
For a smooth immersion $\iota:\Sigma\to M$, we say that $\Sigma$ is uniformly immersed if each $x\in\Sigma$ has a local branch through $\iota(x)$ satisfying (1), with a fixed graph radius and uniform derivative bounds. The branch is the image of a neighborhood of $x$ on which $\iota$ is an embedding; other branches may meet the same ambient ball.

\begin{lemma}[Local equivalence of distance]\label{lem:distance-equivalence}
Let $\Sigma$ be uniformly immersed, with the metric induced from $M$. For every $c>1$ there is $\nu_c>0$, independent of the base point, such that
\[
 d_\Sigma(x,y)\le c\,d_g(\iota(x),\iota(y))\quad\text{if }d_\Sigma(x,y)<\nu_c.
\]
\end{lemma}
\begin{proof}
In the normal graph coordinates, $Dh_x(0)=0$ and $Dh_x$ is uniformly continuous. On a sufficiently small fixed ball the induced metric is as close to its value at the center as desired. Comparing lengths of the straight segment in the tangent chart gives the result. The opposite inequality $d_g\le d_\Sigma$ is automatic.
\end{proof}

\begin{lemma}[Uniformly locally finite cover]\label{lem:cover}
For every sufficiently small fixed $\delta>0$, $M$ has a countable cover $\{B_M(y_j,\delta)\}$ with $d_g(y_j,y_k)\ge\delta$ for $j\ne k$. Every fixed enlargement of these balls has uniformly finite overlap.
\end{lemma}
\begin{proof}
Choose a maximal $\delta$-separated set. Its $\delta$-balls cover $M$. Disjoint balls of radius $\delta/3$ around centers in any fixed enlarged ball have uniformly positive volumes. The volume of the enlarged ball is uniformly bounded by bounded geometry, so the number of centers is bounded.
\end{proof}

\begin{lemma}\label{lem:Sigma-bg}
A uniformly embedded submanifold as in Definition~\ref{def:embedding}, equipped with its induced metric, has bounded geometry.
\end{lemma}
The graph and separation hypotheses give uniform intrinsic coordinate balls and a positive intrinsic injectivity radius. The Gauss equation and its derivatives give bounded intrinsic curvature jets. This is the bounded-geometry submanifold statement in~\cite[Lemma~2.27]{Eldering}.

We use these facts to specify coverings adapted to restriction; compare~\cite[Corollary~2.26]{Eldering}.
\begin{proposition}\label{prop:cover}
There is $\vartheta=\vartheta(\Sigma)>0$ and covers $\{A_j\}$ of $\Sigma$ and $\{B_j\}$ of a fixed neighborhood of $\Sigma$ in $M$, such that $\{B_j\}$ has uniformly finite overlap and
\[
 \{x\in M:d_g(x,A_j)<\vartheta/4\}\subset B_j.
\]
All sets have uniformly controlled local coordinates.
\end{proposition}
\begin{proof}
Take $\vartheta=\frac12\min(\delta(M),\varpi(\Sigma))$ after reducing the radii if necessary. Choose a separated cover $B_M(y_m,\vartheta/8)$ as in Lemma~\ref{lem:cover}. Whenever this ball meets $\Sigma$, choose $x_m\in B_M(y_m,\vartheta/8)\cap\Sigma$, and put
\begin{equation}\label{eq:AjBj}
 A_j=\Sigma(x_{m_j},\vartheta/2),\qquad
 B_j=B_M(y_{m_j},\vartheta).
\end{equation}
The uniqueness of the component in Definition~\ref{def:embedding} shows that the $A_j$ cover $\Sigma$. If $d_g(x,A_j)<\vartheta/4$, then
\[
 d_g(x,y_{m_j})<\vartheta/4+\vartheta/2+\vartheta/8<\vartheta.
\]
The enlarged balls $B_j$ have uniformly finite overlap by Lemma~\ref{lem:cover}.
\end{proof}
Choose a smooth partition $\sum_j\phi_j=1$ on $\Sigma$, with $\supp\phi_j\subset A_j$, and cutoffs $\psi_j\in C_c^\infty(B_j)$ equal to one in the $\vartheta/8$ neighborhood of $A_j$. The sum of the $\psi_j$ and all required fixed-order derivatives are uniformly bounded.

\subsection{Unit-band projection estimates}
We first establish the unit-band estimate under the bounded-geometry hypotheses.
\begin{proposition}\label{prop:unitband}
Under the geometric hypotheses of Theorem~\ref{thm:static-general}, for $2\le q\le\infty$,
\begin{equation}\label{eq:unitband}
\norm{R_\Sigma\1_{[\lambda,\lambda+1]}(P)}{2\to q}
\lesssim \lambda^{\mu(q)}L_{k,q}(\lambda),\qquad \lambda\ge2,
\end{equation}
where $L_{k,q}(\lambda)=(\log\lambda)^{1/2}$ if $(k,q)=(n-2,2)$ or if $n\ge3$, $k=n-1$ and $q=2n/(n-1)$, and $L_{k,q}(\lambda)=1$ otherwise. For $n\ge3$, the hypersurface endpoint logarithm may be removed using~\cite{Hu}.
\end{proposition}
\begin{proof}
Choose a real even $\rho\in\mathcal S(\R)$ with $\rho(0)=1$, with $\widehat\rho$ supported where $\epsilon/4\le |t|\le\epsilon/2$, and $|\rho(s)|\ge1/2$ for $|s|\le1$. Such a choice is obtained from a nonnegative even Fourier bump after taking the fixed $\epsilon$ sufficiently small. Define
\[
 \sigma_\lambda(P)=\rho(\lambda-P)+\rho(\lambda+P).
\]
Its kernel is supported in $d_g(x,y)\le\epsilon$, where $\epsilon<\vartheta/8$, by the exact even wave representation
\begin{equation}\label{eq:unit-wave}
 \sigma_\lambda(P)=\frac1\pi\int_\R
 e^{it\lambda}\widehat\rho(t)\cos(tP)\dd t.
\end{equation}
For sufficiently large $\lambda$, $\sigma_\lambda$ is bounded away from zero on $[\lambda,\lambda+1]$. Thus
\begin{equation}\label{eq:unit-factor}
 \1_{[\lambda,\lambda+1]}(P)=\sigma_\lambda(P)b_\lambda(P),
 \qquad\norm{b_\lambda(P)}{2\to2}\le C,
\end{equation}
where $b_\lambda(s)=\1_{[\lambda,\lambda+1]}(s)/\sigma_\lambda(s)$.

Finite propagation gives
\[
 \phi_jR_\Sigma\sigma_\lambda=\phi_jR_\Sigma\sigma_\lambda\psi_j.
\]
We may follow the local proofs in~\cite[Theorems~1, 3 and~4]{BGT-restriction} to obtain
\begin{equation}\label{eq:local-BGT}
 \norm{\phi_jR_\Sigma\sigma_\lambda\psi_j}{L^2(B_j)\to L^q(A_j)}
 \le C\lambda^{\mu(q)}L_{k,q}(\lambda),
\end{equation}
with $C$ independent of $j$. We recall the part of the argument which verifies this uniformity.

On a normal-coordinate patch the small-time spectral kernel is a finite sum of operators of the form
\begin{equation}\label{eq:local-osc}
 \lambda^{(n-1)/2}\int e^{\pm i\lambda d_g(x,x')}
 a_\lambda(x,x')f(x')\dd x'+\mathcal R_\lambda f(x),
\end{equation}
where $a_\lambda$ has uniform fixed-order derivative bounds and is supported where $d_g(x,x')$ is comparable to a small fixed radius. The local remainder satisfies the required uniform remainder bounds. For the normalized oscillatory part restricted to graph coordinates on $\Sigma$, let $K(z,z')$ be its $TT^*$ kernel. At $|z-z'|\lesssim\lambda^{-1}$,
\[
 |K(z,z')|\lesssim(1+\lambda|z-z'|)^{-(n-1)/2},
 \qquad
 \sup_z\norm{\1_{|z-z'|\lesssim\lambda^{-1}}K(z,z')}{L^{q/2}_{z'}}
 \lesssim\lambda^{-2k/q}.
\]
For dyadic distances $r=2^{-m}$ with $\lambda^{-1}\lesssim r\lesssim\epsilon$, the operators $K_m$ satisfy
\begin{align}
 \norm{K_m}{1\to\infty}&\lesssim(\lambda r)^{-(n-1)/2},\label{eq:dyadic-L1}\\
 \norm{K_m}{2\to2}&\lesssim
 r^k(\lambda r)^{-(n+k-2)/2}
 =2^{-mk}\left(\frac{2^m}{\lambda}\right)^{(n+k-2)/2}.
 \label{eq:dyadic-L2}
\end{align}
The factor $r^k$ comes from rescaling the integration measure. These are the rescaled estimates of~\cite[Proposition~6.3, (6.6)]{BGT-restriction}. Interpolation and dyadic summation, followed by the original factor $\lambda^{(n-1)/2}$, give~\eqref{eq:local-BGT} in the stated ranges. At $(n,k,q)=(2,1,4)$, use $|K(z,z')|\lesssim\lambda^{-1/2}|z-z'|^{-1/2}$ and one-dimensional fractional integration, as in~\cite[Section~3]{BGT-restriction}. This gives $\|K\|_{4/3\to4}\lesssim\lambda^{-1/2}$ without a logarithm.

On the fixed off-diagonal rescaled supports, the distance phase is a controlled perturbation of the Euclidean phase. The bounded metric jets and the uniform graph jets permit the radius $\epsilon$ and all finite-order phase constants to be chosen independently of the patch.

For $q<\infty$, finite overlap and $q\ge2$ give
\begin{align*}
 \norm{R_\Sigma\sigma_\lambda f}{q}
 &\le C\left(\sum_j\norm{\phi_jR_\Sigma\sigma_\lambda\psi_jf}{q}^q\right)^{1/q}\\
 &\le C\lambda^{\mu(q)}L_{k,q}(\lambda)
       \left(\sum_j\norm{\psi_jf}{2}^2\right)^{1/2}
 \le C\lambda^{\mu(q)}L_{k,q}(\lambda)\norm f2.
\end{align*}
For $q=\infty$, use the supremum of the local estimates. Now apply~\eqref{eq:unit-factor}. A bounded interval of positive frequencies is handled by the same local argument with fixed constants.
\end{proof}

\subsection{Proof of Theorem~\ref{thm:static-general}}
We first take $\lambda$ sufficiently large and choose $c>0$ small. Choose a real even Schwartz function $\rho$ with compactly supported Fourier transform and $|\rho|\ge c_1>0$ on a fixed neighborhood of $[-1,1]$. Set $\Xi=|\rho|^2$ and
\begin{equation}\label{eq:positive-window}
 m_{\lambda}(\nu)=\Xi(c\log\lambda\,(\lambda-\nu))+\Xi(c\log\lambda\,(\lambda+\nu)).
\end{equation}
Reducing the fixed $c$ if necessary, $m_{\lambda}\ge c_1^2$ on
$[\lambda,\lambda+(\log\lambda)^{-1}]$. Consequently
\begin{equation}\label{eq:positive-factor}
 \1_{[\lambda,\lambda+(\log\lambda)^{-1}]}(P)
 =m_{\lambda}(P)^{1/2}b_{\lambda}(P),
 \qquad\norm{b_{\lambda}(P)}{2\to2}\le c_1^{-1}.
\end{equation}
It suffices by $TT^*$ to bound $R_\Sigma m_{\lambda}(P)R_\Sigma^*$ by
$C\lambda^{2\mu(q)}/(c\log\lambda)$ from $L^{q'}$ to $L^q$. The exact cosine formula is
\begin{equation}\label{eq:positive-cosine}
 m_{\lambda}(P)=\frac1{\pi c\log\lambda}\int_\R
 e^{it\lambda}\widehat\Xi(t/(c\log\lambda))\cos(tP)\dd t.
\end{equation}

Choose $\Phi\in C_c^\infty(\R)$ equal to one near zero and supported in a sufficiently small fixed interval. Split~\eqref{eq:positive-cosine} into $L_\lambda+G_\lambda$ using $\Phi(t)$ and $1-\Phi(t)$.

\subsubsection{Global estimate}
We obtain the global estimate by interpolation. The Fourier transform of
$(1-\Phi(t))\widehat\Xi(t/(c\log\lambda))$ satisfies
\[
 |\Upsilon_\lambda(s)|\le C_Nc\log\lambda\,(1+|s|)^{-N}.
\]
Decomposition into unit spectral bands and Proposition~\ref{prop:unitband} give
\begin{equation}\label{eq:global-static-L2}
 \norm{R_\Sigma G_\lambda R_\Sigma^*}{2\to2}
 \lesssim\begin{cases}
 \lambda^{2\mu(2)},&k\ne n-2,\\
 \lambda^{2\mu(2)}\log\lambda,&k=n-2.
 \end{cases}
\end{equation}
Insert each unit-band multiplier between the two restriction factors and use its $L^2$ operator norm. The rapidly decreasing band coefficients are summable.

The long-time wave parametrix on the universal cover gives
\begin{equation}\label{eq:global-static-point}
 |G_\lambda(x,y)|\le C\lambda^{(n-1)/2}e^{C_Mc\log\lambda}.
\end{equation}
This is the spectral wave-kernel bound proved in~\cite{HSTZ}, extending the compact argument of~\cite[Section~5]{Chen}. Orbit separation follows from the positive injectivity radius, and volume comparison gives at most $Ce^{CR}$ translates at distance at most $R$, uniformly in the endpoints. Fixed-order Hadamard transport coefficients and their derivatives have at most exponential growth by the Jacobi equations and bounded curvature jets. See~\cite[Appendix~B]{Keeler}. A finite expansion has a residual source supported in a ball of radius $|t|$ on the universal cover, with its required derivatives bounded by $Ce^{C|t|}$. Since this ball has volume at most $Ce^{C|t|}$, the source has the same type of bound in each required fixed Sobolev norm. The zero-data wave energy estimate and uniform local Sobolev evaluation control the continuous remainder by $Ce^{C|t|}$. Applying the usual fixed-order oscillatory integration to the leading wave terms proves~\eqref{eq:global-static-point}.

Interpolate~\eqref{eq:global-static-L2} with the $L^1\to L^\infty$ bound in~\eqref{eq:global-static-point}. If $k\le n-2$, the resulting bound is
\begin{equation}\label{eq:global-static-codim}
 C\lambda^{(n-1)/2+(n-1-2k)/q}
 e^{C_Mc\log\lambda(1-2/q)}(\log\lambda)^{\frac2q\1_{\{k=n-2\}}}.
\end{equation}
For $q>2$ its power of $\lambda$ is strictly less than $2\mu(q)$. If $k=n-1$, the power is
\begin{equation}\label{eq:global-static-hyper}
 \frac{n-1}{2}-\frac{n-2}{q}<2\mu(q),
 \qquad q>\frac{2n}{n-1}.
\end{equation}
Fix a positive number smaller than the appropriate strict power gap and then take $c$ sufficiently small. The remaining positive power of $\lambda$ absorbs both the displayed logarithm and the required factor $c\log\lambda$. This proves
\begin{equation}\label{eq:global-static-done}
 \norm{R_\Sigma G_\lambda R_\Sigma^*}{q'\to q}
 \lesssim \lambda^{2\mu(q)}(c\log\lambda)^{-1}.
\end{equation}

\subsubsection{Local estimate}
The local time cutoff has fixed compact support and uniformly bounded derivatives for $c\log\lambda\ge1$. Hence $L_\lambda$ is $(c\log\lambda)^{-1}$ times a uniformly rapidly decreasing sum of unit-band multipliers centered at $\pm\lambda$. Proposition~\ref{prop:unitband}, with the exponents in Theorem~\ref{thm:static-general}, gives
\begin{equation}\label{eq:local-static-done}
 \norm{R_\Sigma L_\lambda R_\Sigma^*}{q'\to q}
 \lesssim (c\log\lambda)^{-1}\lambda^{2\mu(q)}.
\end{equation}
Combining~\eqref{eq:global-static-done} and~\eqref{eq:local-static-done}, and using~\eqref{eq:positive-factor}, proves the estimate for large $\lambda$. The remaining bounded range follows from Proposition~\ref{prop:unitband}. This proves Theorem~\ref{thm:static-general}.\qed

\subsection{Geodesic segments}\label{subsec:geodesic-spectral}
Throughout this subsection let $(M,g)$ be a smooth bounded-geometry surface with $-b^2\le K_g\le-a^2<0$. Fix $0<\ell<\min(1,\inj(M)/10)$. For the geometric arguments fix a short observed unit-speed geodesic $\gamma:I\to M$, $0\in I$, with $|I|\le\ell$. Longer compact segments are subdivided into such intervals. Let $\widetilde\gamma$ be a lift to the universal cover $(\widetilde M,\widetilde g)$, with lifted arclength parameter, and put $x_0=\widetilde\gamma(0)$. Write $\mathcal G$ for the deck group and $\mathcal L$ for the complete geodesic extending $\widetilde\gamma$. Set
\begin{equation}\label{eq:distance-alpha}
 d_\alpha(r,u)=d_{\widetilde g}(\widetilde\gamma(r),\alpha\widetilde\gamma(u)),
 \qquad
 w_\alpha(r,u)=\left(\frac{d_\alpha(r,u)}{J_\alpha(d_\alpha(r,u))}\right)^{1/2}.
\end{equation}
Here $J_\alpha$ is the scalar normal Jacobi solution along the connecting geodesic, normalized by $J_\alpha(0)=0$ and $J_\alpha'(0)=1$. Its dependence on $(r,u)$ is suppressed. The diagonal value of $w_\alpha$ is defined by smooth continuation.

The universal cover is Cartan--Hadamard. Jacobi comparison gives, for every connector in~\eqref{eq:distance-alpha},
\begin{equation}\label{eq:Jacobi-comparison}
 \frac{\sinh(a d_\alpha)}a\le J_\alpha(d_\alpha)\le\frac{\sinh(b d_\alpha)}b,
 \qquad 0<w_\alpha\le1.
\end{equation}
Orbit points are uniformly separated because $\inj(M)>0$. Packing fixed-radius balls and using the curvature lower bound gives
\begin{equation}\label{eq:orbit-count}
 \#\{\alpha\in\mathcal G:d_{\widetilde g}(x_0,\alpha x_0)\le D\}\le Ce^{CD},\qquad D\ge1.
\end{equation}
The constants are uniform in the base point.

Define
\begin{equation}\label{eq:tube-set}
 \mathcal G_{\mathrm{tube}}
 =\{\alpha:\dist_{\widetilde g}(\alpha x_0,\mathcal L)\le 2\ell+3\},
 \qquad \mathcal G_{\mathrm{osc}}=\mathcal G\setminus\mathcal G_{\mathrm{tube}}.
\end{equation}
Nearest-point projection onto $\mathcal L$ is distance decreasing. The part of its fixed-radius tube meeting $B_{\widetilde g}(x_0,D)$ can therefore be covered by $O(1+D)$ balls of a fixed radius. Each such ball contains uniformly boundedly many orbit points. Hence
\begin{equation}\label{eq:tube-count}
 \#\{\alpha\in\mathcal G_{\mathrm{tube}}:d_{\widetilde g}(x_0,\alpha x_0)\le D\}
 \le C(1+D).
\end{equation}

The next lemma follows the distance-derivative argument of Blair~\cite[Lemmas~2.2--2.4]{Blair}. We give the endpoint calculation on $\widetilde M$.

\begin{lemma}\label{lem:distance-geometry}
Let $x(u)$ and $y(r)$ be unit-speed geodesic segments of bounded length in $\widetilde M$, and set $d(r,u)=d_{\widetilde g}(y(r),x(u))$. Let $D\ge3$ and $B>0$, and assume $3\le d(r,u)\le D$. If the angle at $x(u)$ between the connecting geodesic and each of the two tangent directions of $x$ is at least $e^{-BD}$ throughout the parameter rectangle, then
\begin{equation}\label{eq:finite-geometry}
 \max\{|d_{ru}|,|d_{ruu}|\}\ge C_B^{-1}e^{-C_BD},
 \qquad \norm d{C^4}\le C_Be^{C_BD}.
\end{equation}
All constants depend only on fixed geometric bounds, $B$, and the lengths of the parameter intervals.
\end{lemma}
\begin{proof}
Along the connector from $x(u)$ to $y(r)$ choose a parallel oriented orthonormal frame $T,N$, with $T$ its unit tangent, and write
\[
 x'(u)=a_xT(0)+b_xN(0),\qquad
 y'(r)=a_yT(d)+b_yN(d).
\]
Let $J$ denote the normalized scalar normal Jacobi solution along this connector. Varying $x(u)$ with $y(r)$ fixed gives a normal Jacobi field $v$ with $v(0)=b_x$, $v(d)=0$. Its Wronskian with $J$ is constant, so $v'(d)=-b_x/J(d)$. Consequently
\begin{equation}\label{eq:Jacobi-identities}
 d_{ru}=-\frac{b_xb_y}{J(d)},\qquad
 \partial_ub_y=\frac{a_yb_x}{J(d)}.
\end{equation}
Differentiating once more gives
\begin{equation}\label{eq:third-identity}
 d_{ruu}=-\frac{a_yb_x^2}{J(d)^2}
 -b_y\partial_u\left(\frac{b_x}{J(d)}\right).
\end{equation}
The Jacobi equation is $J''+K_{\widetilde g}J=0$ along the connector. In bounded-geometry charts, the geodesic and Jacobi equations have uniformly bounded coefficients and fixed-order derivatives. Endpoint differentiation gives inhomogeneous linear systems. Their source terms involve lower-order derivatives. On an interval of length at most $D$, the fundamental matrices and their inverses are bounded by $Ce^{CD}$. Induction and variation of constants give the same type of bound for each fixed derivative order. The endpoint map is invertible because the cover has no conjugate points. Its normal inverse is controlled by $J(d)^{-1}$, and Jacobi comparison gives
\[
 J(d)^{-1}\le a\,\bigl(\sinh(ad)\bigr)^{-1}\le C,
 \qquad d\ge3.
\]
Normal coordinates give uniform bounds on each fixed distance interval bounded away from zero. Consequently
\[
 \norm d{C^4}+\left|\partial_u(b_x/J(d))\right|\le Ce^{CD}.
\]
The same induction gives the fixed-order Hadamard coefficient bounds used below, as in~\cite[Appendix~B]{Keeler}.

The angle hypothesis implies $|b_x|\ge C^{-1}e^{-BD}$. If $|b_y|\ge e^{-B_1D}$, the first identity in~\eqref{eq:Jacobi-identities} and the upper bound for $J$ give the required lower bound for $|d_{ru}|$. Choose $B_1$ sufficiently large. If the preceding alternative fails, then $|b_y|<e^{-B_1D}$ and hence $|a_y|\ge1/2$. The first term of~\eqref{eq:third-identity} then has size at least $C^{-1}e^{-2(B+b)D}$. Increasing $B_1$ if necessary makes the second term at most half this size. This proves the lower bound for $|d_{ruu}|$.
\end{proof}

We will use Blair's estimate~\cite[Lemma~2.4]{Blair}. If $\phi\in C^4$, $b_0\in C^1$ on a fixed rectangle, $|\phi_{ruu}|\ge\epsilon$ whenever $|\phi_{ru}|\le\epsilon$, and $A\ge\max(1,\norm{\phi_{ru}}{C^2})$, then for $\Lambda\ge1$
\begin{equation}\label{eq:Blair}
 \left\|\int e^{i\Lambda\phi(r,u)}b_0(r,u)f(u)\dd u\right\|_2
 \le C\Lambda^{-1/4}\epsilon^{-1}A^{5/4}\norm{b_0}{C^1}\norm f2.
\end{equation}
After decreasing $\epsilon$ by a fixed factor,~\eqref{eq:finite-geometry} gives the required hypothesis. The polynomial dependence in~\eqref{eq:Blair} preserves the exponential geometric bounds.

\subsubsection{The critical estimate}\label{subsec:critical}
Theorem~\ref{thm:static-general} gives the log-scale estimate on curves for $q>4$. We now prove the critical estimate for $2\le q\le4$ on compact geodesic segments.
\begin{proposition}\label{prop:critical-log}
Let $M$ be a complete surface of bounded geometry with $-b^2\le K_g\le-a^2<0$, and let $\gamma$ be a fixed compact geodesic segment. For $\lambda\ge2$,
\begin{equation}\label{eq:critical-static}
 \norm{R_{\gamma}\1_{[\lambda,\lambda+(\log\lambda)^{-1}]}(P)}{2\to q}
 \lesssim_{q,\gamma}\lambda^{\mug(q)}(\log\lambda)^{-1/2},
 \qquad2\le q\le\infty.
\end{equation}
In the range $2\le q\le4$, the frequency power is $1/4$.
\end{proposition}
\begin{proof}
It suffices to work on a fixed short embedded parameter interval. A longer compact segment is covered by finitely many such pieces, and their output norms can be summed. For $q>4$, extend each short segment to a smooth closed embedded curve in its coordinate neighborhood and apply Theorem~\ref{thm:static-general}. For $q=4$, choose $c>0$ sufficiently small and use the positive multiplier~\eqref{eq:positive-window} and its local/global decomposition.

The small-time local kernel satisfies
\begin{equation}\label{eq:critical-local}
 |L_\lambda(\gamma(r),\gamma(u))|
 \lesssim (c\log\lambda)^{-1}\lambda^{1/2}|r-u|^{-1/2}.
\end{equation}
Fractional integration gives its $L^{4/3}\to L^4$ bound $C(c\log\lambda)^{-1}\lambda^{1/2}$.

Split the deck transformations as in~\eqref{eq:tube-set}. Denote the corresponding principal restricted sums by $G_\lambda^{\mathrm{tube}}$ and $G_\lambda^{\mathrm{osc}}$.

The finite Hadamard expansion and stationary phase, as in~\cite{ChenSogge,Blair}, write each long-wave term at distance $d_\alpha(r,u)\ge1$ as
\begin{equation}\label{eq:static-leading}
 \frac{\lambda^{1/2}}{c\log\lambda}
 \sum_\pm e^{\pm i\lambda d_\alpha}
 J_\alpha(d_\alpha)^{-1/2}b_{\alpha,\pm}(r,u)
 +\mathcal R_\alpha(r,u).
\end{equation}
The amplitudes $b_{\alpha,\pm}$ come from the leading Hadamard term and have bounded size and at most exponentially growing fixed-order derivatives. The higher Hadamard terms and the smooth remainder are included in $\mathcal R_\alpha$. After summation over the finitely many effective translates, the summed remainder is $O(\lambda^{Cc})$. Uniform transport and the infinite-volume residual-source estimate are justified as in the proof of~\eqref{eq:global-static-point}. The zero-distance and finitely many short-distance terms obey~\eqref{eq:critical-local} with smooth errors controlled by the same estimate.

Jacobi comparison~\eqref{eq:Jacobi-comparison} and the counting bound~\eqref{eq:tube-count} give
\[
 \norm{G_\lambda^{\mathrm{tube}}}{4/3\to4}
 \lesssim\frac{\lambda^{1/2}}{c\log\lambda}.
\]
Indeed, the exponential decay of $J_\alpha(d_\alpha)^{-1/2}$ makes the principal kernels absolutely summable on the fixed observation square.

For $\alpha\in\mathcal G_{\mathrm{osc}}$, the point $\alpha\widetilde\gamma(u)$ has distance greater than $\ell+3$ from $\mathcal L$. Suppose $d_\alpha\le D$ on the parameter rectangle. Compare the connector from $\widetilde\gamma(r)$ with either tangent ray of $\mathcal L$ at the same length. Upper Jacobi comparison bounds their endpoint distance by $Ce^{CD}$ times their initial angle. The angle at $\widetilde\gamma(r)$ is therefore at least $C^{-1}e^{-CD}$. Lemma~\ref{lem:distance-geometry}, with the endpoints interchanged, gives
\begin{equation}\label{eq:tube-transversality}
 \max\{|\partial_r\partial_u d_\alpha|,
        |\partial_r^2\partial_u d_\alpha|\}
 \ge C^{-1}e^{-CD},\qquad
 \norm{d_\alpha}{C^4}\le Ce^{CD}.
\end{equation}
For the spectral kernel, $D\lesssim c\log\lambda$. Apply~\eqref{eq:Blair} to the spatial adjoint, so the repeated derivative is in its input variable. After summation over the effective deck transformations,
\[
 \norm{G_\lambda^{\mathrm{osc}}}{2\to2}\lesssim\lambda^{1/4+Cc},
 \qquad
 \norm{G_\lambda^{\mathrm{osc}}}{1\to\infty}\lesssim\lambda^{1/2+Cc}.
\]
Interpolation yields
\[
 \norm{G_\lambda^{\mathrm{osc}}}{4/3\to4}
 \lesssim(\lambda^{1/2})^{1/2}(\lambda^{1/4})^{1/2}\lambda^{Cc}
 =\lambda^{3/8+Cc}.
\]
Choose the finite expansion order first and then $c>0$ sufficiently small. This and the smooth remainder are bounded by $C\lambda^{1/2}/(c\log\lambda)$. The $TT^*$ factorization proves~\eqref{eq:critical-static} at $q=4$. For $2\le q<4$, use the finite length of $\gamma$.
\end{proof}

\section{Narrow spectral windows on asymptotically hyperbolic surfaces}\label{sec:AH-static}
\subsection{Geometry of the ends and local smoothing}
Let $(M,g)$ be an even asymptotically hyperbolic surface with bounded geometry and curvature pinched below $0$. We follow~\cite{HSTZ} to decompose
\begin{equation}\label{eq:core-end}
 M=M_0\cup M_\infty,
\end{equation}
where $M_0$ is compact and $M_\infty$ consists of sufficiently far-out parts of the ends.
Let $p(x,\xi)=|\xi|_g$ on $T^*M\setminus0$. Denote its Hamilton flow by $\Phi_t$ and the bundle projection by $\pi:T^*M\to M$. Define
\begin{equation}\label{eq:incoming-outgoing}
 \Lambda_\pm=\{(x,\xi)\in S^*M:\pi\Phi_t(x,\xi)\not\longrightarrow\infty
             \text{ as }t\to\pm\infty\}.
\end{equation}
The trapped set in phase space is $\mathfrak K=\Lambda_+\cap\Lambda_-$, and its spatial projection is $\pi(\mathfrak K)$. Choose $M_0$ to contain a neighborhood of this compact projection. Sufficiently far out, the ends are geodesically outgoing: after crossing an outer level in the outgoing direction, a ray does not return.

For each end, the construction in~\cite[(2.35) and Lemma~2.4]{HSTZ} gives a smooth complete simply connected even asymptotically hyperbolic surface $(Y,g_0)$ with curvature pinched below zero. After enlarging $M_0$, the metrics agree on the corresponding component of $M_\infty$. The construction parameter can be chosen so that $-\Delta_{g_0}$ has no $L^2$ eigenvalues and the threshold $1/4$ is regular.

Each escaping tail identifies with a ray of a complete background geodesic $\Gamma$. A compact parameter interval accounts for the remaining part of the observed geodesic. The compact-segment estimate is Proposition~\ref{prop:critical-log}. We first obtain the corresponding estimates on the auxiliary surfaces.

Put $h=\lambda^{-1}$. For high-frequency data $f_\lambda=\1_{[c_1,c_2]}(hP)f_\lambda$, with $0<c_1<c_2$ fixed,~\cite[(2.46)--(2.47)]{HSTZ} gives
\begin{align}
 \norm{\chi U(t)f_\lambda}{L^2(\R\times M)}
 &\lesssim (h\log(1/h))^{1/2}\norm{f_\lambda}2,
 &&\chi\in C_c^\infty(M),\label{eq:LS-interior}\\
 \norm{\widetilde\chi U(t)f_\lambda}{L^2(\R\times M)}
 &\lesssim h^{1/2}\norm{f_\lambda}2,
 &&\widetilde\chi\in C_c^\infty(M_\infty).\label{eq:LS-exterior}
\end{align}
On the nontrapping backgrounds~\eqref{eq:LS-exterior} holds for every compact cutoff. These estimates are on the whole time axis. They follow from the physical cutoff resolvent bounds $\lambda^{-1}\log\lambda$ and $\lambda^{-1}$, respectively, by Stone's formula and the spectral-measure $TT^*$ argument. These bounds use only high energies.

The derivative estimates follow by local semiclassical elliptic regularity. For nested compact cutoffs with $\widetilde\chi_1=1$ near $\supp\chi_1$,
\begin{equation}\label{eq:static-local-elliptic}
 \norm{(1-h^2\Delta_g)^{(m+2)/2}(\chi_1v)}2
 \le C_m\left(\norm{(1-h^2\Delta_g)^{m/2}\bigl(\widetilde\chi_1(-h^2\Delta_g)v\bigr)}2
       +\norm{\widetilde\chi_1v}2\right).
\end{equation}
Apply the smoothing estimate to $(-h^2\Delta_g)^jU(t)f_\lambda=U(t)(-h^2\Delta_g)^jf_\lambda$ and iterate. If $B$ is a differential operator of order $m$ with compactly supported coefficients in a lossless region, then
\begin{equation}\label{eq:derivative-LS}
 \norm{BU(t)f_\lambda}{L^2_{t,x}}\lesssim h^{1/2-m}\norm{f_\lambda}2.
\end{equation}

For a compact cutoff $\chi_1$ supported in a lossless region, we also have the spectral-mass bound of~\cite[Lemma~2.9]{HSTZ}:
\begin{equation}\label{eq:input-window}
 \norm{\1_{[\nu,\nu+\eta]}(P)\chi_1}{2\to2}
 =\norm{\chi_1\1_{[\nu,\nu+\eta]}(P)}{2\to2}
 \lesssim\eta^{1/2},\qquad \nu\sim\lambda,\quad 0<\eta\le1.
\end{equation}
Indeed the cutoff density for $(-\Delta_g)$ is $O(\lambda^{-1})$, so the cutoff density for $P$ is $O(1)$. Integrate over the frequency interval and take a $TT^*$ square root. The same argument allows $\eta$ larger than one while the interval remains in a fixed high-energy annulus.

\subsection{Spectral estimates on the auxiliary surface}\label{subsec:background-spectral}
Let $(Y,g_0)$ be a smooth complete simply connected asymptotically hyperbolic surface of bounded geometry with $K_{g_0}\le0$. Set
\[
 P_0=\sqrt{-\Delta_{g_0}},\qquad U_0(t)=e^{it\Delta_{g_0}},\qquad
 Q_0=\sqrt{-\Delta_{g_0}-1/4},\qquad
 S(\kappa)=\frac{dE_{Q_0}(\kappa)}{d\kappa}.
\]
Here $\Pczer$ is the absolutely continuous spectral projection for $-\Delta_{g_0}$. The shifted square root $Q_0$ and its spectral resolution are defined on its range, and $S(\kappa,x,y)$ is the scalar density kernel with respect to Riemannian volume. Let $\Gamma:\R\to Y$ be a complete unit-speed geodesic. Write $\kappa>0$ for the spectral parameter of $Q_0$ and $\nu>0$ for that of $P_0$.
 By Cartan--Hadamard, each complete geodesic in $Y$ is globally minimizing. Thus $Y$ is nontrapping, and the high-energy resolvent construction in~\cite[Theorem~38]{CHI} applies. Stone's formula gives
\[
 S(\kappa)=\frac{\kappa}{\pi i}\left[
 \bigl(-\Delta_{g_0}-\tfrac14-(\kappa+i0)^2\bigr)^{-1}
 -\bigl(-\Delta_{g_0}-\tfrac14-(\kappa-i0)^2\bigr)^{-1}
 \right].
\]
The ordinary resolvent is $\kappa^{-2}$ times the inverse of $-\kappa^{-2}\Delta_{g_0}-(4\kappa^2)^{-1}-1$. The factor in Stone's formula is therefore $\kappa^{-1}/(\pi i)$ on this semiclassical inverse. The common pseudodifferential terms cancel. The spectral equation and propagation of Lagrangian regularity extend the resulting structure across the diagonal, as in the proof of~\cite[Theorem~3.3]{CH}. This high-energy argument does not require a regular threshold.

Since $Y$ is simply connected and $K_{g_0}\le0$, its points are joined by a unique geodesic and there are no conjugate points. The spectral phases away from the diagonal are therefore $\pm d_{g_0}(x,y)$. The argument of~\cite[Proposition~6.4 and Section~6]{CH} applies with the identity microlocal partition. This gives the high-energy bounds in~\cite[Theorems~1.3 and~1.5]{CH}. Put $d=d_{g_0}(x,y)$. For sufficiently large $\kappa$,
\begin{equation}\label{eq:CH-point}
 |S(\kappa,x,y)|\le C
 \begin{cases}
 \kappa(1+\kappa d)^{-1/2},&d\le1,\\
 \kappa^{1/2}e^{-d/2},&d\ge1.
 \end{cases}
\end{equation}
The same construction gives, for every fixed $N$,
\begin{equation}\label{eq:CH-osc}
 S(\kappa,x,y)=\sum_\pm e^{\pm i\kappa d}a_\pm(\kappa,x,y)
 +E(\kappa,x,y),\qquad |E(\kappa,x,y)|\le C_N\kappa^{-N}e^{-d/2},
\end{equation}
where, for each fixed nonnegative integer $j$,
\begin{equation}\label{eq:CH-symbol}
 |\partial_\kappa^ja_\pm(\kappa,x,y)|\le C_j
 \begin{cases}
 \kappa^{1-j}(1+\kappa d)^{-1/2},&d\le1,\\
 \kappa^{1/2-j}e^{-d/2},&d\ge1.
 \end{cases}
\end{equation}
We also use the derivative bound
\begin{equation}\label{eq:CH-whole-derivative}
 |\partial_\kappa S(\kappa,x,y)|\le C(1+\kappa d)^{1/2},\qquad d\le1.
\end{equation}

The prefactor in~\cite[(1.14)]{CH} is $\kappa$ in dimension two. Estimate~\cite[(1.15)]{CH} then gives the first line of~\eqref{eq:CH-symbol}. For $d\ge1$, use~\cite[(1.16), Proposition~3.4 and (9.21)]{CH}. The four residual terms in~\cite[(1.14)]{CH} decay to arbitrary order in $\kappa^{-1}$. Their boundary weights are bounded by $Ce^{-d/2}$, which gives the remainder in~\eqref{eq:CH-osc}. For every fixed $j,N\ge0$, the same expansion gives
\begin{equation}\label{eq:CH-residual-derivatives}
 |\partial_\kappa^j E(\kappa,x,y)|
 \le C_{j,N}\kappa^{-N}e^{-d/4}.
\end{equation}
Indeed the residual coefficients in~\cite[(1.14)]{CH} are smooth on the compactified double space and vanish to every order in $\kappa^{-1}$. Differentiating their boundary oscillations introduces logarithms of the boundary products. For either product $z$,~\cite[Proposition~3.4 and (9.21)]{CH} gives $z\le Ce^{-d}$, and $z^{1/2}|\log z|^j\le C_jz^{1/4}$. Estimate~\eqref{eq:CH-whole-derivative} follows from~\cite[(1.18)]{CH}. The density kernel is taken with respect to Riemannian volume. Restriction to $\Gamma$ gives an integral kernel acting against arclength $du$.

For the unshifted operator $P_0$, the spectral \emph{density} is, for $\nu>1/2$,
\begin{equation}\label{eq:shift-density}
 \frac{dE_{P_0}}{d\nu}(\nu)
 =\frac{\nu}{\kappa}S(\kappa),\qquad \kappa=\sqrt{\nu^2-1/4}.
\end{equation}
For $\nu\ge1$, $\kappa\asymp\nu$ and $1\le\nu/\kappa\le2/\sqrt3$, so at high frequency the unshifted density obeys~\eqref{eq:CH-point} with $\nu$ in place of $\kappa$. The endpoints of $[\lambda,\lambda+\eta]$ map to $\sqrt{\lambda^2-1/4}$ and $\sqrt{(\lambda+\eta)^2-1/4}$.

If the threshold is regular,~\cite[Theorem~1.5, (1.17)]{CH} also gives
\begin{equation}\label{eq:CH-low}
 |S(\kappa,x,y)|\le C
 \begin{cases}
 \kappa^2,&d\le1,\\
 \kappa^2d(1+\kappa d)^{-1}e^{-d/2},&d\ge1,
 \end{cases}
 \qquad 0<\kappa\le1.
\end{equation}
Cartan--Hadamard gives
\begin{equation}\label{eq:line-distance}
 d_{g_0}(\Gamma(r),\Gamma(u))=|r-u|.
\end{equation}
We next restrict these kernels to $\Gamma$.

\begin{lemma}[Restriction of the background spectral density]\label{lem:CH-curve}
For a complete unit-speed geodesic $\Gamma$ on $(Y,g_0)$, uniformly for shifted frequencies $\kappa\gtrsim1$ and unshifted frequencies $\nu\ge1$, we have, for $2\le q\le\infty$,
\begin{align}
 \norm{R_\Gamma S(\kappa)R_\Gamma^*}{L^{q'}(\R)\to L^q(\R)}
 &\lesssim_q\kappa^{2\mug(q)},\label{eq:CH-curve-shifted}\\
 \norm{R_\Gamma\frac{dE_{P_0}}{d\nu}(\nu)R_\Gamma^*}{L^{q'}(\R)\to L^q(\R)}
 &\lesssim_q\nu^{2\mug(q)}.\label{eq:CH-curve-unshifted}
\end{align}
In particular,
\begin{equation}\label{eq:CH-Schur}
 \norm{R_\Gamma S(\kappa)R_\Gamma^*}{L^2(\R)\to L^2(\R)}\lesssim\kappa^{1/2}.
\end{equation}
\end{lemma}
\begin{proof}
By~\eqref{eq:CH-point} and~\eqref{eq:line-distance}, for any fixed $0<c<1/2$,
\begin{equation}\label{eq:CH-majorant}
 |S(\kappa,\Gamma(r),\Gamma(u))|
 \le C\kappa(1+\kappa|r-u|)^{-1/2}e^{-c|r-u|}.
\end{equation}
Direct integration gives
\[
 \left\|\kappa(1+\kappa|v|)^{-1/2}e^{-c|v|}\right\|_{L^{q/2}(\R_v)}
 \lesssim_q
 \begin{cases}\kappa^{1/2},&2\le q<4,\\
 \kappa^{1-2/q},&4<q\le\infty.
 \end{cases}
\]
Young's inequality proves~\eqref{eq:CH-curve-shifted} for $q\ne4$. At $q=4$, the kernel is bounded by $C\kappa^{1/2}|r-u|^{-1/2}$, and we apply the one-dimensional Hardy--Littlewood--Sobolev inequality. In particular, the $q=2$ estimate follows from
\[
 \int_0^1\kappa(1+\kappa v)^{-1/2}\dd v
 +\int_1^\infty\kappa^{1/2}e^{-v/2}\dd v
 \lesssim\kappa^{1/2}.
\]
The bounded positive-frequency ranges follow from~\cite[Remark~1.4]{CH}. Equation~\eqref{eq:shift-density} then gives~\eqref{eq:CH-curve-unshifted}. The constants are uniform in $\Gamma$ by~\eqref{eq:line-distance}.
\end{proof}

We will also use the global trace bound
\begin{equation}\label{eq:global-Sobolev-trace}
 \norm{R_\Gamma w}{L^q(\R)}\le C_q\norm w{H^2(Y)},
 \qquad w\in H^2(Y),\quad 2\le q\le\infty.
\end{equation}
Divide $\Gamma$ into uniformly short intervals and apply the local $H^2\to L^q$ trace estimate. The enlarged coordinate balls have bounded overlap by~\eqref{eq:line-distance}. The $\ell^q$ norm of the local $H^2$ norms is at most their $\ell^2$ norm. For $q=\infty$ use the supremum of the local estimates. This proves~\eqref{eq:global-Sobolev-trace} uniformly in $\Gamma$.

\subsubsection{Spectral windows}
The spectral-density bounds give the following estimate for arbitrarily small windows.
\begin{proposition}\label{prop:background-window}
Let $(Y,g_0)$ be one of the auxiliary surfaces constructed above and let $\Gamma$ be a complete unit-speed geodesic. For $\lambda\ge1$, $0<\eta\le1$ and $2\le q\le\infty$,
\begin{equation}\label{eq:background-window}
 \norm{R_{\Gamma}\1_{[\lambda,\lambda+\eta]}(P_0)}{2\to q}
 \lesssim\lambda^{\mug(q)}\eta^{1/2}.
\end{equation}
\end{proposition}
\begin{proof}
Let $\Pi=\1_{[\lambda,\lambda+\eta]}(P_0)$. Since $\Pi$ is an orthogonal projection,
\[
 (R_\Gamma\Pi)(R_\Gamma\Pi)^*
 =\int_\lambda^{\lambda+\eta}
 R_\Gamma\frac{dE_{P_0}}{d\nu}(\nu)R_\Gamma^*\dd\nu.
\]
By~\eqref{eq:CH-curve-unshifted}, the integral has $L^{q'}\to L^q$ norm at most $C\lambda^{2\mug(q)}\eta$. Taking a $TT^*$ square root proves~\eqref{eq:background-window}. The same argument applies to windows of any length $\eta$ lying in a fixed annulus $\nu\asymp\lambda$.
\end{proof}

We also need the spectral-window estimate with a microlocal cutoff.
\begin{lemma}\label{lem:microlocal-window}
Let $A_h$ have a properly supported compact-frequency semiclassical kernel satisfying, uniformly for each sufficiently large fixed $N$,
\begin{equation}\label{eq:A-kernel}
 |A_h(x,z)|\le C_Nh^{-2}(1+d_{g_0}(x,z)/h)^{-N},
 \qquad A_h(x,z)=0\quad\text{if }d_{g_0}(x,z)>D_0.
\end{equation}
For $\nu\sim h^{-1}$ and any interval $I\subset[c/h,C/h]$ of length $\eta$,
\begin{equation}\label{eq:A-window}
 \norm{R_{\Gamma}A_h\1_I(P_0)}{2\to q}
 \le C_q\nu^{\mug(q)}\eta^{1/2},\qquad2\le q\le\infty.
\end{equation}
Constants are uniform over families satisfying~\eqref{eq:A-kernel}.
\end{lemma}
\begin{proof}
Compose the density on both sides with $A_h$ and $A_h^*$. For its double integral, the triangle inequality gives
\[
 e^{-c d_{g_0}(z,w)}\le e^{2cD_0}e^{-c d_{g_0}(x,y)},
\]
while
\[
 (1+\nu d_{g_0}(z,w))^{-1/2}\le
 (1+\nu d_{g_0}(x,y))^{-1/2}
 (1+\nu d_{g_0}(x,z)+\nu d_{g_0}(w,y))^{1/2}.
\]
Since $\nu h\sim1$, the extra polynomial factor is uniformly integrable against the two kernels~\eqref{eq:A-kernel}. The composed density is therefore bounded by
\[
 C\nu(1+\nu d_{g_0}(x,y))^{-1/2}e^{-c d_{g_0}(x,y)}.
\]
After restriction to $\Gamma$, the proof of Lemma~\ref{lem:CH-curve} applies with $\nu$ in place of $\kappa$. Integrating over $I$ bounds $R_{\Gamma}A_h\1_I(P_0)A_h^*R_{\Gamma}^*$ by $C\nu^{2\mug(q)}\eta$. Taking a square root proves the result.
\end{proof}

Uniformly locally finite sums of properly supported order-zero operators with uniform compact-frequency symbols satisfy~\eqref{eq:A-kernel}. Finite overlap makes the kernel constants uniform.

\subsection{Half-localized resolvent estimates}
Fix $2<q<\infty$ and write $\mu=\mug(q)$. We prove the two resolvent estimates used in the gluing argument. Choose a real $\beta_0\in C_c^\infty((1/4,4))$ equal to one on a neighborhood of $[1/2,2]$.

\begin{proposition}[The background half-localized resolvent]\label{prop:bg-halfres}
On the same auxiliary surface, let $\widetilde\chi_0\in C_c^\infty(Y)$ and let $\Gamma$ be a complete geodesic. For $\lambda\ge\lambda_0$, $0<\eta\le1$, and $2<q<\infty$,
\begin{equation}\label{eq:bg-halfres}
 \norm{R_{\Gamma}(-\Delta_{g_0}-\lambda^2-i\lambda\eta)^{-1}\widetilde\chi_0}{2\to q}
 \lesssim_q\lambda^{\mug(q)-1}.
\end{equation}
The same bound holds with $+i\lambda\eta$ in place of $-i\lambda\eta$.
\end{proposition}
\begin{proof}
Put $h=\lambda^{-1}$ and write
\[
 \lambda_\delta+i\delta=\sqrt{\lambda^2+i\lambda\eta},
 \qquad \lambda_\delta\asymp\lambda,\quad\delta\asymp\eta,
\]
with positive real and imaginary parts. Then $\lambda_\delta-\lambda=O(\eta^2/\lambda)$. In the spectral decompositions below, the windows are centered at $\lambda_\delta$.

The even wave formula is
\begin{equation}\label{eq:res-wave}
 (-\Delta_{g_0}-(\lambda_\delta+i\delta)^2)^{-1}
 =\frac{i}{\lambda_\delta+i\delta}\int_0^\infty
 e^{it\lambda_\delta-t\delta}\cos(tP_0)\dd t.
\end{equation}
The part with $1-\beta_0(P_0/\lambda)$ is bounded by $C\lambda^{\mu-1}$, uniformly in $\delta$: decompose into unit bands, use Proposition~\ref{prop:background-window}, and sum $\nu^\mu/|\nu^2-(\lambda_\delta+i\delta)^2|$ away from the annulus. At very low energy use the fixed-energy background kernel bound. The series at infinity converges since $\mu<1$.

Let $b\in C_c^\infty((1/2,2))$ be a dyadic partition on the positive time axis. For $j\in\Z$, set
\begin{equation}\label{eq:Tj}
 T_j=\frac{i}{\lambda_\delta+i\delta}\int_0^\infty
 b(2^{-j}t)e^{it\lambda_\delta-t\delta}\cos(tP_0)\beta_0(P_0/\lambda)\dd t.
\end{equation}
All estimates below are uniform for $0<\delta\lesssim1$. Differentiation of $e^{-t\delta}$ on a dyadic support only produces bounded factors of $(2^j\delta)^Ne^{-c 2^j\delta}$.

\paragraph{Spectral bounds for an individual time scale.}
For $\nu$ in the annulus, integration by parts in $t$ gives
\begin{equation}\label{eq:Tj-symbol}
 |T_j(\nu)|\le C_N\lambda^{-1}2^j
 \bigl(1+2^j|\nu-\lambda_\delta|\bigr)^{-N},
 \qquad 2^j\ge\lambda^{-1}.
\end{equation}
The reflected half-wave has $\nu+\lambda_\delta$ in place of $\nu-\lambda_\delta$ and obeys the same bound. On the intervals of width $2^{-j}$ about $\lambda_\delta+m2^{-j}$, Proposition~\ref{prop:background-window}, or Lemma~\ref{lem:microlocal-window} with $A_h$ retained, and input orthogonality show that
\begin{equation}\label{eq:individual-noinput}
 \norm{R_{\Gamma}A_h T_j}{2\to q}
 \lesssim \lambda^{\mu-1}2^{j/2},
 \qquad \lambda^{-1}\le 2^j\lesssim1.
\end{equation}
Here and below $A_h$ may be omitted. For $2^j\ge1$, the compact input cutoff supplies a second factor $2^{-j/2}$ through~\eqref{eq:input-window} on the background. Hence
\begin{equation}\label{eq:individual-withinput}
 \norm{R_{\Gamma}A_hT_j\widetilde\chi_0}{2\to q}\lesssim\lambda^{\mu-1},
 \qquad 2^j\ge1.
\end{equation}
Explicitly, the three factors on the $m$th window are
\[
 \lambda^\mu 2^{-j/2},\qquad
 C_N\lambda^{-1}2^j(1+|m|)^{-N},\qquad
 C2^{-j/2},
\]
and they are summed in $m$.

For $2^j\le\lambda^{-1}$ the sum of the multipliers is $O(\lambda^{-2})$ on the annulus. The dyadic restriction norm is at most $C\lambda^{\mu+1/2}$, by unit-band input orthogonality, so this part is $O(\lambda^{\mu-3/2})$. Summing~\eqref{eq:individual-noinput} over $\lambda^{-1}\le 2^j\le C_0$, with $C_0$ fixed, gives $O(\lambda^{\mu-1})$.

\paragraph{The large logarithmic tail.}
The spectral-measure argument in~\cite{HSTZ}, applied to the smooth frequency cutoff in~\eqref{eq:Tj}, gives
\begin{equation}\label{eq:exp-tail}
 \norm{T_j\widetilde\chi_0}{2\to L^\infty(Y)}
 \le C_N\left(\lambda^{1/2}e^{-2^j/16}+(\lambda 2^j)^{-N}\right),
 \qquad 2^j\ge1.
\end{equation}
In the unshifted variable $\nu$, the principal spectral phases are $\pm\kappa(\nu)d$, where $\kappa(\nu)=\sqrt{\nu^2-1/4}$. For $d\ge2^{j-2}$, integrate the multiplier bound~\eqref{eq:Tj-symbol} against the exponentially decreasing density. For $d<2^{j-2}$ and $t\in[2^{j-1},2^{j+1}]$, the phases $\pm t\nu\pm\kappa(\nu)d$ have $\nu$-derivative of size at least $c2^j$, since $\kappa'(\nu)=1+O(\lambda^{-2})$ on the band. Integration by parts in $\nu$ gives the second term of~\eqref{eq:exp-tail}. At $d\lesssim\lambda^{-1}$ use the complete-kernel derivative estimates in~\cite[(1.18)]{CH}. The residual part is treated by~\eqref{eq:CH-residual-derivatives}. All bounds are uniform in $\delta>0$. Compact support of $\widetilde\chi_0$ converts the scalar bound to $L^2\to L^\infty$.

An operator $A_h$ satisfying~\eqref{eq:A-kernel} is uniformly bounded on $L^\infty$, so~\eqref{eq:exp-tail} holds after $A_h$ and restriction. The proof of~\eqref{eq:individual-withinput} also applies at $q=2$ and gives $C\lambda^{-3/4}$, uniformly in $j$. Interpolation therefore gives
\begin{equation}\label{eq:correct-tail-interpolation}
 \norm{R_{\Gamma}A_hT_j\widetilde\chi_0}{2\to q}
 \le C_q\lambda^{1/2-5/(2q)}
 e^{-(1-2/q)2^j/16}+C_{q,N}(\lambda 2^j)^{-N},
\end{equation}
after increasing the order used in~\eqref{eq:exp-tail} for the last term. For each fixed $q>2$, the sum over $2^j\ge C_q\log\lambda$ is $O(\lambda^{\mu-1})$ if $C_q$ is large enough.

Choose now a sufficiently small fixed $c_0>0$ for the microlocal argument below. There are at most
\[
 2+\left\lceil\log_2(C_q/c_0)\right\rceil
\]
dyadic scales between $c_0\log\lambda$ and $C_q\log\lambda$. Consequently~\eqref{eq:individual-withinput} handles this entire intervening range.

\paragraph{Directional decomposition in the remaining range.}
It remains to sum $C_0\le 2^j\le c_0\log\lambda$. Split the compact input cutoff into finitely many sufficiently small coordinate pieces, and fix one supported near $y_0$. Put
\[
 S_\lambda=\{x\in\Gamma:C_0/8\le d_{g_0}(x,y_0)\le8c_0\log\lambda\}.
\]
For $x\notin S_\lambda$, the source radius and the choice of $C_0$ give a fixed separation between $d_{g_0}(x,y)$ and the retained wave times. The preceding nonstationary calculation therefore gives, for each fixed $N$,
\[
 \left\|\1_{\Gamma\setminus S_\lambda}
 R_\Gamma\sum_{C_0\le2^j\le c_0\log\lambda}T_j\widetilde\chi_0
 \right\|_{2\to\infty}\le C_N\lambda^{-N}.
\]
The spectral-window estimates give the global $L^2(\Gamma)$ bound
$C\lambda^{-3/4}(1+\log\lambda)$ for this sum. Interpolation yields
\[
 C_{q,N}\lambda^{-N(1-2/q)-3/(2q)}
 (1+\log\lambda)^{2/q}
\]
on the complement of $S_\lambda$. For fixed $q>2$, choose $N$ so that this is $O(\lambda^{-2})$.

By~\eqref{eq:line-distance}, the part of $\Gamma$ inside a ball of radius $8c_0\log\lambda$ has length at most $16c_0\log\lambda$. Thus a fixed neighborhood of $S_\lambda$ is covered by $O(1+\log\lambda)$ normal coordinate balls with uniformly bounded overlap. Choose these balls from a fixed uniformly controlled cover along $\Gamma$. Their partition of unity gives a cutoff $\theta_\lambda=1$ near $S_\lambda$, with uniformly bounded derivatives and support in a fixed enlargement.

In a chart centered at $x_\iota$, let $\xi_\iota$ be the terminal unit cotangent covector of the geodesic from $y_0$ to $x_\iota$. Use a fixed-aperture angular cutoff equal to one near $\xi_\iota$. Hessian comparison bounds the change of the terminal direction as the output varies in the chart. Varying the source changes this direction by a normal Jacobi field with terminal derivative bounded by $J(d)^{-1}$. Since $d\ge C_0/16$, the source and chart radii can be fixed so that both variations are smaller than the angular margin.

Choose a real $\beta_1\in C_c^\infty((0,\infty))$ equal to one on $\supp\beta_0$. Split the local symbol of $\theta_\lambda\beta_1(hP_0)$ by these angular cutoffs, through a sufficiently large fixed order. The resulting properly supported operators satisfy
\begin{equation}\label{eq:direction-partition}
 \theta_\lambda\beta_1(hP_0)=A_{h,0}+A_{h,1}+\mathcal E_N,
 \qquad
 \norm{\mathcal E_N}{H^{-2}\to H^2}
 \le C_N\lambda^{-N}(1+\log\lambda).
\end{equation}
To obtain the remainder bound, first truncate the even wave formula for $\beta_1(hP_0)$ at a fixed wave time shorter than the chart margin. The discarded scalar multiplier and its polynomially weighted derivatives are $O(h^N)$ for every fixed $N$, as in~\cite[(2.23)--(2.24)]{HSTZ}. The spectral theorem gives the same bound from $H^{-2}$ to $H^2$. The retained kernel has fixed propagation radius and a local symbol expansion to any prescribed finite order. Summing its local Sobolev remainders over the $O(1+\log\lambda)$ charts gives~\eqref{eq:direction-partition}. Each $A_{h,i}$ satisfies~\eqref{eq:A-kernel} with uniform constants, by finite overlap. In local coordinates its phase is $\lambda\langle x-z,\xi\rangle$, with a compact frequency symbol and fixed proper support.

The middle-time sum satisfies
\[
 \left\|\sum_{C_0\le2^j\le c_0\log\lambda}T_j\widetilde\chi_0\right\|_{2\to2}
 \lesssim\lambda^{-1}\log\lambda.
\]
Hence~\eqref{eq:global-Sobolev-trace} bounds its composition with $\mathcal E_N$ by
$C_{q,N}\lambda^{-N-1}(1+\log\lambda)^2$ from $L^2(Y)$ to $L^q(\Gamma)$.

For $A_{h,1}$, we use the scalar-kernel calculation in~\cite[proof of (2.113), particularly (2.120)--(2.129)]{HSTZ}. In this middle-time range, integration by parts in wave time makes the term containing $1-\beta_0(P_0/\lambda)$ an $O(\lambda^{-N})$ operator from $H^{-2}$ to $H^2$, for every fixed $N$. Thus we may omit $\beta_0$ in the kernel calculation. Finite propagation restricts the distance to $d_{g_0}(z,y)\le C2^j$. The radial frequency integral may then be localized to a fixed annulus, since its complement is nonstationary in wave time. The finite Hadamard terms, composed with a coordinate piece of $A_{h,1}$, have phases $\lambda\Phi_\pm$, where
\[
 \Phi_\pm=\frac{\lambda_\delta}{\lambda}t
       +d_{g_0}(z,y)\zeta_1\pm t|\zeta|
       +\langle x-z,\xi\rangle,
 \qquad |\zeta|\asymp1.
\]
Here $t\asymp2^j$. The plus sign has $\partial_t\Phi_+\ge c>0$. For the minus sign, simultaneous stationarity in $t$ and $\zeta$ requires
\[
 |\zeta|=\lambda_\delta/\lambda,\qquad
 \zeta/|\zeta|=(1,0),\qquad t=d_{g_0}(z,y).
\]
Away from a fixed neighborhood of this set, integration by parts in $t$ or $\zeta$ applies. On that neighborhood, put $\sigma=\zeta_1>0$. The angular margin of $A_{h,1}$ and the source and chart radii give
\[
 |\sigma\,d_z\bigl(d_{g_0}(z,y)\bigr)-\xi|_{g_0}\ge c,
\]
so integration by parts in $z$ applies. Both terms in this difference are cotangent covectors. For the incoming time phase $-t\lambda_\delta$, the stationary branch instead has $\zeta_1<0$ and the same separation holds on the retained cone $A_{h,0}$.

The higher radial terms have the same phases. On $t\asymp2^j$, all fixed-order amplitude and phase bounds are at most $C_Ne^{C_N2^j}$. Time derivatives of the damped cutoff obey the uniform bounds stated after~\eqref{eq:Tj}. The continuous wave remainder is controlled by the source and wave-energy argument used for~\eqref{eq:global-static-point}. Fixing the differentiation and expansion orders, the preceding integrations by parts give the pointwise bound
$C_N\lambda^{-N}e^{C_N2^j}\le C_N\lambda^{-N+C_Nc_0}$, uniformly in $\delta>0$.
Compact source support converts this to an $L^2$ input bound. The $O(1+\log\lambda)$ observed length and the dyadic sum give the bound
\[
 C_{q,N}\lambda^{-N+C_Nc_0}(1+\log\lambda)^{1+1/q}.
\]
Choose $N$ and the expansion order, and then $c_0$ small enough. We obtain
\begin{equation}\label{eq:wrong-direction}
 \left\|R_\Gamma A_{h,1}
 \sum_{C_0\le2^j\le c_0\log\lambda}T_j\widetilde\chi_0\right\|_{2\to q}
 \le C_q\lambda^{-2}.
\end{equation}

For the other cone, write the middle-time sum as the full outgoing resolvent minus the already estimated short-time, intervening logarithmic, and long-time sums. Those three estimates hold with $A_{h,0}$ retained. It remains to bound
\[
 R_{\Gamma}A_{h,0}(-\Delta_{g_0}-(\lambda_\delta+i\delta)^2)^{-1}\beta_0(P_0/\lambda)\widetilde\chi_0.
\]
Let $\mathcal R_{0,\pm}=(-\Delta_{g_0}-(\lambda_\delta\pm i\delta)^2)^{-1}$, with scalar multipliers $r_{0,\pm}(\nu)=(\nu^2-(\lambda_\delta\pm i\delta)^2)^{-1}$. On a spectral interval of width $\delta$ centered at $\lambda_\delta+m\delta$, the difference satisfies
\[
 |r_{0,+}(\nu)-r_{0,-}(\nu)|
 \le C(\lambda\delta)^{-1}(1+|m|)^{-2}.
\]
Apply Lemma~\ref{lem:microlocal-window} on the output and~\eqref{eq:input-window} on the input. The resulting sum is
\begin{equation}\label{eq:retained-Q-jump}
 \lambda^\mu\delta^{1/2}(\lambda\delta)^{-1}\delta^{1/2}
 \sum_m(1+|m|)^{-2}\lesssim\lambda^{\mu-1}.
\end{equation}
We may therefore replace the outgoing resolvent by the incoming one. The negative stationary direction of the incoming phase is separated from the symbol of $A_{h,0}$, as above, so its middle-time contribution is bounded by $C_q\lambda^{-2}$. The short, intervening and long ranges have the same bounds for either sign with $A_{h,0}$ retained. Together with the separated-frequency estimates above, this proves~\eqref{eq:bg-halfres}.
\end{proof}

For the fixed exponent $2<q<\infty$, write $\mu=\mug(q)$. The background spectral theorem also gives the unrestricted-input estimate
\begin{equation}\label{eq:bg-fullres}
 \norm{R_{\Gamma}(-\Delta_{g_0}-\lambda^2-i\lambda\eta)^{-1}}{2\to q}
 \lesssim\lambda^{\mu-1}\eta^{-1/2}.
\end{equation}
Indeed use spectral intervals of width $\eta$, their output factor $\lambda^\mu\eta^{1/2}$, the multiplier bound $C(\lambda\eta)^{-1}(1+|m|)^{-1}$, and Cauchy--Schwarz in $m$. The off-annular part is handled by unit bands.

\begin{proposition}[The interior half-localized resolvent]\label{prop:original-halfres}
Let $\chi_1\in C_c^\infty(M)$ and let $\widetilde\chi_1\in C_c^\infty(M_\infty)$ have support with a fixed margin away from the trapped projection. For a fixed compact geodesic segment $\gamma$ and with $\mu=\mug(q)$,
\begin{equation}\label{eq:original-halfres}
 \norm{R_{\gamma}\chi_1(-\Delta_g-\lambda^2-i\lambda/\log\lambda)^{-1}
       \widetilde\chi_1}{2\to q}
 \lesssim_{q,\gamma}\lambda^{\mu-1},\qquad2<q<\infty.
\end{equation}
Without the input cutoff, Proposition~\ref{prop:critical-log} gives
\begin{equation}\label{eq:original-fullres}
 \norm{R_{\gamma}\chi_1(-\Delta_g-\lambda^2-i\lambda/\log\lambda)^{-1}}{2\to q}
 \lesssim_{q,\gamma}\lambda^{\mu-1}(\log\lambda)^{1/2}.
\end{equation}
\end{proposition}
\begin{proof}
Formula~\eqref{eq:original-fullres} follows from log-width spectral intervals, Cauchy--Schwarz and input $L^2$ orthogonality, just as for~\eqref{eq:bg-fullres}. Smooth multiplication by $\chi_1$ only localizes the observed compact segment. For~\eqref{eq:original-halfres}, remove separated spatial frequencies by unit-band estimates. In the central annulus, the compactly supported incoming/outgoing decomposition of~\cite[Lemma~2.13]{HSTZ} gives
\begin{equation}\label{eq:B-decomp}
 \beta_0(P/\lambda)\widetilde\chi_1=B_++B_-+\mathcal R,
 \qquad\norm{\mathcal R}{2\to2}=O(\lambda^{-1}).
\end{equation}
A finite sum of coordinate pieces is understood. On the compact normalized frequency support over $\supp\widetilde\chi_1$, each covector escapes in at least one time direction, since this support is disjoint from $\pi(K)$. Choose the angular partition with supports compactly contained in the corresponding escape sets. The uniform escape statement~\cite[(2.1)]{HSTZ} gives a fixed $C_0$ such that slightly enlarged microsupports of $B_+$ and $B_-$ avoid a neighborhood of $\supp\chi_1$ for all $t\ge C_0/2$ and all $t\le-C_0/2$, respectively. Each $B_\pm$ has compact input and output support in a slightly enlarged exterior region. The error in~\eqref{eq:B-decomp} is controlled by~\eqref{eq:original-fullres}. Contributions of $B_\pm$ outside the central frequency annulus are bounded by the preceding unit-band argument.

Write $E_m=\1_{[\lambda_\delta+m/\log\lambda,\lambda_\delta+(m+1)/\log\lambda]}(P)$, where $\lambda_\delta$ is the real part of $\sqrt{\lambda^2+i\lambda/\log\lambda}$. Then $\lambda_\delta\sim\lambda$, and the imaginary part is comparable to $(\log\lambda)^{-1}$. The two estimates used on these intervals are
\begin{equation}\label{eq:halfres-two-windows}
 \norm{R_{\gamma}\chi_1E_m}{2\to q}
 \lesssim\lambda^\mu (\log\lambda)^{-1/2},\qquad
 \norm{E_mB_\pm}{2\to2}\lesssim (\log\lambda)^{-1/2}.
\end{equation}
The second follows by inserting a compact exterior cutoff on the output of $B_\pm$ and using~\eqref{eq:input-window}.

Let $\mathcal R_\pm=(-\Delta_g-(\lambda_\delta\pm i\delta)^2)^{-1}$ with $\delta=\Im\sqrt{\lambda^2+i\lambda/\log\lambda}$. For a difference of the two resolvent signs, the scalar multiplier on $E_m$ is bounded by
$C\lambda^{-1}\log\lambda(1+|m|)^{-2}$. The two factors in~\eqref{eq:halfres-two-windows} cancel both square roots of $\log\lambda$. Hence
\begin{equation}\label{eq:original-jump}
 \norm{R_{\gamma}\chi_1(\mathcal R_+-\mathcal R_-)B_\pm}{2\to q}
 \lesssim\lambda^{\mu-1}.
\end{equation}
We may therefore choose the resolvent sign which is compatible with the escape direction of each $B_\pm$.

Use the smooth dyadic partition $b(2^{-j}t)$ from~\eqref{eq:Tj}. The scales $2^j\le C_0$ have total norm at most $C\lambda^{\mu-1}$. For $\lambda^{-1}\le2^j\le1$, the multiplier has size $O(\lambda^{-1}2^j)$ on a window of width $2^{-j}$. The output factor is $C\lambda^\mu2^{-j/2}$, and input orthogonality gives $C\lambda^{\mu-1}2^{j/2}$. The smaller scales contribute $O(\lambda^{\mu-3/2})$, as in the background proof. Each of the finitely many scales from $1$ to $C_0$ contributes at most $C\lambda^{\mu-1}$.

For the scales $2^j\ge c_0\log\lambda$, the time cutoff is
\[
 \sum_{\{j:2^j\ge c_0\log\lambda\}}b(2^{-j}t).
\]
After $t=(\log\lambda)v$, its product with $e^{-t\delta}$ is a uniformly bounded family in $\mathcal S(\R_v)$, extended by zero for $v<0$. Here $\delta\log\lambda\asymp1$. Fourier transformation therefore bounds the multiplier on $E_m$ by
\[
 C_N\lambda^{-1}\log\lambda(1+|m|)^{-N}.
\]
Equation~\eqref{eq:halfres-two-windows} gives $C\lambda^{\mu-1}$ after summation.

For $C_0<2^j<c_0\log\lambda$, the escape direction makes the returning phase nonstationary. We use the localized pointwise calculation of~\cite{HSTZ} in the adjoint expression. For $B_+$ and the outgoing sign, it has the form
\[
 \mathcal B_j=B_+^*\int b(2^{-j}t)e^{-it\lambda_\delta-t\Im\sqrt{\lambda^2+i\lambda/\log\lambda}}
        \cos(tP)\chi_1\dd t,
\]
up to the scalar $O(\lambda^{-1})$ and a smooth compact-frequency cutoff. A stationary branch of this adjoint kernel would give a forward geodesic from the enlarged microsupport of $B_+$ to $\supp\chi_1$ at time $t\asymp2^j>C_0/2$, contrary to the escape property. The fixed angular and spatial margins persist on the coordinate supports. On the universal cover, the signed-phase argument above therefore gives a bound $C_N\lambda^{-N}e^{C_N2^j}$ for each effective translate. The exponential orbit count is absorbed in $C_N$. Both the output support of $B_\pm^*$ and the input support of $\chi_1$ are compact. Integrating this scalar bound over the ambient output and the compact observed input gives
\[
 \norm{\mathcal B_jR_{\gamma}^*}{L^{q'}(\gamma)\to L^2(M)}
 \le C_N\lambda^{-N+C_Nc_0}.
\]
The case of $B_-$ uses the conjugate time phase. Fix $N$ and the symbol expansion order, and then reduce $c_0$. The sum over the $O(\log\log\lambda)$ intermediate scales is $O(\lambda^{-2})$. A left spectral projection is an $L^2$ contraction on this output. This proves~\eqref{eq:original-halfres}.
\end{proof}

\subsection{Proof of Theorem~\ref{thm:AH-static}}\label{subsec:AH-proof}
Fix $2<q<\infty$ and write $\mu=\mug(q)$. Choose $\chi_0\in C_c^\infty(M)$ equal to one on $M_0$, and write
\[
 \chi_\infty=1-\chi_0=\sum_e\chi_{\infty,e}.
\]
Each $\chi_{\infty,e}$ is supported in one end, with derivatives in the lossless region, and has its own auxiliary surface $Y_e$. The compact part of the observation has finite parameter length. We estimate it and the finitely many end contributions separately.

\subsubsection{The exterior part}
It suffices to estimate the smooth spectral operator
\begin{equation}\label{eq:smooth-small-window}
 \rho\left(\frac{(-\Delta_g)-\lambda^2}{\lambda\eta}\right),
\end{equation}
on input $f_\lambda$ spectrally supported in a fixed annulus $P\sim\lambda$. Here $\rho$ is an even Schwartz function with compact Fourier support and bounded below on a fixed interval containing $[0,3]$. A bounded input factor recovers the sharp frequency interval $[\lambda,\lambda+\eta]$, since
\[
 0\le\frac{\nu^2-\lambda^2}{\lambda\eta}\le3
 \quad\text{for }\lambda\ge1,\ \lambda\le\nu\le\lambda+\eta,\ 0<\eta\le1.
\]
Put
\[
 w(t)=\frac{\lambda\eta}{2\pi}\widehat\rho(\lambda\eta t),
 \qquad u(t)=U(t)f_\lambda.
\]
Fourier inversion gives
\begin{equation}\label{eq:spectral-time}
 \rho((-\Delta_g-\lambda^2)/(\lambda\eta))f_\lambda
 =\int_\R w(t)e^{it\lambda^2}u(t)\dd t.
\end{equation}
The weight has compact time support and satisfies
\begin{equation}\label{eq:weight-bounds}
 \norm w1\lesssim1,\quad \norm w2\lesssim(\lambda\eta)^{1/2},\quad
 \norm w\infty\lesssim\lambda\eta,\quad \norm{w'}1\lesssim\lambda\eta.
\end{equation}

Fix an end $e$ and write $(Y,g_0)$ for its auxiliary surface. Multiplication by $\chi_{\infty,e}$, identification of the common end and extension by zero define $\mathcal J:L^2(M)\to L^2(Y)$. Put
\[
 v(t)=\mathcal Ju(t),\qquad
 \mathcal C_1=\Delta_{g_0}\mathcal J-\mathcal J\Delta_g.
\]
The coefficients of $\mathcal C_1$ are compactly supported in the lossless region. Since $(i\partial_t+\Delta_{g_0})v=\mathcal C_1u$, Duhamel's formula gives
\begin{equation}\label{eq:exterior-Duhamel}
 v(t)=U_0(t)\mathcal Jf_\lambda-iZ(t),\qquad
 Z(t)=\int_0^tU_0(t-s)\mathcal C_1u(s)\dd s.
\end{equation}
For the homogeneous contribution, Schwartz decay gives
\[
 \int_1^\infty
 \left|\rho\!\left(\frac{\nu^2-\lambda^2}{\lambda\eta}\right)\right|^2
 \nu^{2\mu}\dd\nu\lesssim\lambda^{2\mu}\eta.
\]
On $\lambda/2\le\nu\le2\lambda$, use the change of variable $(\nu^2-\lambda^2)/(\lambda\eta)$. The complement contributes at most $C_N\eta^{2N}\lambda^{2\mu+1-2N}$ for fixed sufficiently large $N$. Below $\nu=1$,~\eqref{eq:CH-low} bounds the restricted density integral by $C_N(\eta/\lambda)^{2N}$. Thus~\eqref{eq:CH-curve-unshifted} and $TT^*$ bound the homogeneous contribution to~\eqref{eq:spectral-time} by $C\lambda^\mu\eta^{1/2}\norm{f_\lambda}2$.

The same identity yields
\[
 Z(t)=i\bigl(\mathcal JU(t)f_\lambda-U_0(t)\mathcal Jf_\lambda\bigr).
\]
Unitarity on the two surfaces therefore gives
\begin{equation}\label{eq:Z-uniform}
 \sup_{t\in\R}\norm{Z(t)}2
 \le2\norm{\mathcal J}{2\to2}\norm{f_\lambda}2.
\end{equation}

Since $(i\partial_t+\Delta_{g_0})Z=i\mathcal C_1u$, integration by parts gives
\begin{equation}\label{eq:exterior-exact}
\begin{aligned}
 (-\Delta_{g_0}-\lambda^2-i\lambda\eta)
 \int_\R w(t)e^{it\lambda^2}Z(t)\dd t
 &=-i\int_\R(w'+\lambda\eta w)e^{it\lambda^2}Z(t)\dd t\\
 &\quad-i\int_\R w e^{it\lambda^2}\mathcal C_1u(t)\dd t.
\end{aligned}
\end{equation}
The first source has $L^2$ norm $O(\lambda\eta)\norm{f_\lambda}2$ by~\eqref{eq:weight-bounds} and~\eqref{eq:Z-uniform}. Apply~\eqref{eq:bg-fullres}. The second source has compact support and norm at most
\[
 \norm w2\norm{\mathcal C_1u}{L^2_{t,x}}
 \lesssim (\lambda\eta)^{1/2}\lambda^{1/2}\norm{f_\lambda}2
 =\lambda\eta^{1/2}\norm{f_\lambda}2.
\]
Apply Proposition~\ref{prop:bg-halfres} to this source. Both contributions are bounded by
$C\lambda^\mu\eta^{1/2}\norm{f_\lambda}2$. Thus
\begin{equation}\label{eq:exterior-done}
 \norm{R_\gamma\chi_{\infty,e}\rho((-\Delta_g-\lambda^2)/(\lambda\eta))f_\lambda}{q}
 \lesssim\lambda^\mu\eta^{1/2}\norm{f_\lambda}2.
\end{equation}
There are only finitely many ends and two observed tails. The same proof applies to each of their background realizations.

\subsubsection{The compact part}
Fix a real $\zeta\in C_c^\infty((-2,2))$ with $\sum_{j\in\Z}\zeta(t-j)=1$.
First suppose $0<\eta\le(\log\lambda)^{-1}$. Put
\begin{align*}
 u_j(t)&=\zeta(\lambda t/\log\lambda-j)\chi_0u(t),\\
 F_j&=i\lambda(\log\lambda)^{-1}\zeta'(\lambda t/\log\lambda-j)\chi_0u
       +\zeta(\lambda t/\log\lambda-j)[\Delta_g,\chi_0]u.
\end{align*}
Then $(i\partial_t+\Delta_g)u_j=F_j$. Each $u_j$ is compactly supported in time and space. With the same weight $w$ as in~\eqref{eq:spectral-time}, set
\[
 s_j=\int_\R w(t)e^{it\lambda^2}u_j(t)\dd t,
 \qquad z_\lambda=\lambda^2+i\lambda/\log\lambda.
\]
Since $u_j$ is compactly supported in time, integration by parts gives
\begin{equation}\label{eq:whole-block}
 (-\Delta_g-z_\lambda)s_j
 =-i\int_\R(w'+\lambda(\log\lambda)^{-1}w)e^{it\lambda^2}u_j(t)\dd t
   -\int_\R w e^{it\lambda^2}F_j(t)\dd t.
\end{equation}
Write the right side as $G_j^{(t)}+G_j^{(x)}$, where $G_j^{(x)}$ contains the spatial commutator and $G_j^{(t)}$ contains the remaining terms.

Because $\eta\log\lambda\le1$,
\[
 \norm{w'}\infty+\lambda(\log\lambda)^{-1}\norm w\infty
 \lesssim\lambda^2\eta/\log\lambda.
\]
Let $I_j=[(j-2)\lambda^{-1}\log\lambda,(j+2)\lambda^{-1}\log\lambda]$. Cauchy--Schwarz on $I_j$ gives
\begin{align}
 \norm{G_j^{(t)}}2
 &\lesssim\lambda^{3/2}\eta(\log\lambda)^{-1/2}
          \norm{\chi_0u}{L^2(I_j\times M)},\label{eq:Gtime}\\
 \norm{G_j^{(x)}}2
 &\lesssim\lambda^{1/2}\eta(\log\lambda)^{1/2}
          \norm{[\Delta_g,\chi_0]u}{L^2(I_j\times M)}.\label{eq:Gspace}
\end{align}
Choose $\chi_1=1$ on $\supp\chi_0$ and a compact exterior cutoff $\widetilde\chi_1=1$ on the coefficients of $[\Delta_g,\chi_0]$. Since $s_j$ itself is supported in $\supp\chi_0$, we may apply $R_\gamma\chi_1$ after solving~\eqref{eq:whole-block}. Use~\eqref{eq:original-fullres} for~\eqref{eq:Gtime} and~\eqref{eq:original-halfres} for~\eqref{eq:Gspace}. Finite overlap and~\eqref{eq:LS-interior}--\eqref{eq:derivative-LS} give
\begin{align}
 \sum_j\norm{R_\gamma\chi_1(-\Delta_g-z_\lambda)^{-1}G_j^{(t)}}q^2
 &\lesssim\lambda^{2\mu}\eta^2\log\lambda\norm{f_\lambda}2^2,\label{eq:time-square}\\
 \sum_j\norm{R_\gamma\chi_1(-\Delta_g-z_\lambda)^{-1}G_j^{(x)}}q^2
 &\lesssim\lambda^{2\mu}\eta^2\log\lambda\norm{f_\lambda}2^2.\label{eq:space-square}
\end{align}
For example, the first square sum is
\[
 (\lambda^{\mu-1}(\log\lambda)^{1/2})^2
 (\lambda^{3/2}\eta(\log\lambda)^{-1/2})^2
 (\lambda^{-1}\log\lambda)\norm{f_\lambda}2^2
 =\lambda^{2\mu}\eta^2\log\lambda\norm{f_\lambda}2^2.
\]
The second uses $\norm{[\Delta_g,\chi_0]u}{L^2_{t,x}}^2\lesssim\lambda\norm{f_\lambda}2^2$ and gives the same expression.

The support of $w$ has length $O((\lambda\eta)^{-1})$. Thus at most $C(\eta\log\lambda)^{-1}$ blocks contribute. The triangle inequality followed by Cauchy--Schwarz in this finite collection yields
\begin{align}
 \norm{R_\gamma\sum_js_j}q
 &\lesssim (\eta\log\lambda)^{-1/2}\lambda^\mu\eta(\log\lambda)^{1/2}\norm{f_\lambda}2\notag\\
 &\lesssim\lambda^\mu\eta^{1/2}\norm{f_\lambda}2.\label{eq:compact-done}
\end{align}
By~\eqref{eq:spectral-time}, $\sum_js_j=\chi_0\rho((-\Delta_g-\lambda^2)/(\lambda\eta))f_\lambda$.

Together~\eqref{eq:exterior-done} and~\eqref{eq:compact-done} prove the high-frequency theorem for $\eta\le (\log\lambda)^{-1}$, after inserting the bounded input factor described before~\eqref{eq:smooth-small-window}. If $(\log\lambda)^{-1}\le\eta\le1$, divide the window into $O(\eta\log\lambda)$ intervals of width at most $(\log\lambda)^{-1}$. The triangle inequality and \emph{input} $L^2$ orthogonality give
\[
 \sum_{j\le C\eta\log\lambda}\lambda^\mu (\log\lambda)^{-1/2}\norm{\Pi_jf}2
 \le C\lambda^\mu\eta^{1/2}\norm f2.
\]
This completes the high-frequency proof.

\subsubsection{Bounded positive energies and the endpoint $q=\infty$}
For $1\le\lambda\le\lambda_0$, choose a fixed frequency cutoff supported in $(3/4,2\lambda_0+2)$ and equal to one on a neighborhood of $[1,\lambda_0+1]$. The localized limiting absorption principle away from the threshold, as used in~\cite[proof of (2.45)]{HSTZ}, and local elliptic regularity give uniformly bounded spectral densities on compact subsets of $M$. Integration over a window of length $\eta$ and $TT^*$ give the compact-segment bound $C\eta^{1/2}$.

On the background, the compact positive-frequency spectral kernels satisfy exponential distance bounds for each fixed number of frequency derivatives~\cite[Remark~1.4, Theorem~1.3]{CH}. For the dyadic wave operators in~\eqref{eq:Tj}, restricted to this fixed band, and a fixed $\widetilde\chi_0\in C_c^\infty(Y)$, the proof of~\eqref{eq:exp-tail} gives
\[
 \|T_j\widetilde\chi_0\|_{2\to\infty}\le C_N(e^{-2^j/C}+2^{-jN}),\qquad 2^j\ge1,
\]
whereas the fixed-band trace bound and the scalar multiplier estimate give $\|R_\Gamma T_j\widetilde\chi_0\|_{2\to2}\le C2^j$. Interpolation bounds the polynomial tail by
\[
 C_{q,N}2^{j\{2/q-N(1-2/q)\}}.
\]
It is summable when $N>2/(q-2)$, and the exponential term is summable as well. This gives a uniform half-localized resolvent bound for every fixed $q>2$, including $q=\infty$. The short-time sum has a bounded multiplier on this fixed band. Outside the fixed band the resolvent is uniformly bounded from $L^2$ to $H^2$, so~\eqref{eq:global-Sobolev-trace} handles the remaining frequencies. The original and background local smoothing bounds on the fixed band follow from limiting absorption away from the threshold~\cite[proof of (2.45)]{HSTZ}. The exact identity~\eqref{eq:Z-uniform} also holds for these data. In~\eqref{eq:exterior-exact}, the two source norms are $O(\eta)$ and $O(\eta^{1/2})$. The corresponding resolvent bounds are $O(\eta^{-1/2})$ and $O(1)$. This proves the exterior bound $C\eta^{1/2}$ and completes the bounded-frequency argument.

At high frequency and $q=\infty$, apply the ambient spectral projection estimate of~\cite[Theorem~1.2]{HSTZ}. Restriction decreases the supremum norm, and the bound is $\lambda^{1/2}\eta^{1/2}=\lambda^{\mug(\infty)}\eta^{1/2}$. This completes Theorem~\ref{thm:AH-static}.\qed

\section{Proof of Theorem~\ref{thm:log}}\label{sec:logproof}
Let $(M,g)$ satisfy the hypotheses of Theorem~\ref{thm:log}. Use the geometric notation of Section~\ref{subsec:geodesic-spectral} for a short segment $\gamma:I\to M$, and set $h=\lambda^{-1}$. We begin with the local estimate in physical time.

\subsection{Local Schr\"odinger restriction}\label{sec:local}
On a surface of smooth bounded geometry, the localized integer-order graph norms are uniformly equivalent to coordinate Sobolev norms. Point evaluation is uniformly bounded on $H^2(M)$, and restriction to a fixed compact smooth curve maps $H^2(M)$ to $L^3$.
\begin{lemma}\label{lem:local}
On a complete surface of smooth bounded geometry, let $\gamma:I\to M$ be a sufficiently short unit-speed geodesic segment. If $\psi\in C_c^\infty((0,\infty))$, then, for every fixed $A>0$,
\begin{equation}\label{eq:local-main}
 \norm{R_\gamma U(t)\psi(hP)f}{L^3([0,Ah]\times\gamma)}
 \le C_A\norm f2.
\end{equation}
The constants are uniform over such segments. By time translation, the same estimate holds with $[0,Ah]$ replaced by any of its translates.
\end{lemma}
\begin{proof}
We use the local propagator in Tacy~\cite[Propositions~4.1 and~4.2]{Tacy}. Rescale physical time by $h$ and write
\[
 W_h(t)f=R_\gamma e^{ith\Delta_g}\psi(hP)f.
\]
On a sufficiently short fixed $t$-interval, we claim
\begin{align}
 \norm{W_h(t)W_h(s)^*}{L^1(I)\to L^\infty(I)}
 &\lesssim h^{-1}(h+|t-s|)^{-1},\label{eq:local-1}\\
 \norm{W_h(t)W_h(s)^*}{L^2(I)\to L^2(I)}
 &\lesssim h^{-1/2}(h+|t-s|)^{-1/2}.\label{eq:local-2}
\end{align}
The frequency-localized semiclassical propagator has a local oscillatory representation
\begin{equation}\label{eq:local-param}
 (2\pi h)^{-2}\int e^{i(\phi(t,s,x,\xi)-y\cdot\xi)/h}
 b(t,s,x,\xi,h)\dd\xi,
\end{equation}
up to an error which is $O(h^N)$ between any two fixed local Sobolev spaces after taking sufficiently many transport terms. The amplitude is supported in a fixed frequency annulus, and
\[
 \phi(s,s,x,\xi)=x\cdot\xi,\qquad
 \partial_t\phi+|\partial_x\phi|_g^2=0.
\]
To localize the kernel, use the even cosine functional calculus for the frequency multiplier. Insert a wave-time cutoff supported in a fixed interval shorter than the coordinate-chart margin. Choose $t_1>0$ so small that all stationary wave times for $|t-s|\le t_1$ lie where this cutoff is one. On the discarded part, integration by parts in the spectral variable and then in wave time gives an $O(h^N)$ operator from $H^{-2}(M)$ to $H^2(M)$ for every fixed $N$. Point evaluation gives the same scalar-kernel bound. Finite propagation keeps the retained wave kernel in the coordinate neighborhood, where~\eqref{eq:local-param} applies. The chart and symbol bounds are uniform by smooth bounded geometry.

For $|t-s|\le h$, integration by parts in the tangential frequency shows that the kernel of $W_h(t)W_h(s)^*$ satisfies
\[
 |[W_h(t)W_h(s)^*](r,u)|\le C_Nh^{-2}(1+|r-u|/h)^{-N}.
\]
The supremum is $O(h^{-2})$, and both Schur integrals are $O(h^{-1})$. These are~\eqref{eq:local-1} and~\eqref{eq:local-2} in this range.

Suppose $h\le|t-s|\le t_1$. The stationary region has $|r-u|\lesssim|t-s|$. The frequency Hessian in~\eqref{eq:local-param}, after division by $t-s$, is uniformly definite. Two-dimensional stationary phase gives the kernel factor $(h|t-s|)^{-1}$ and the critical phase
\[
 \frac{d_g(\gamma(r),\gamma(u))^2}{4(t-s)}
 =\frac{(r-u)^2}{4(t-s)}.
\]
The segment is shorter than the injectivity scale. Divide it into intervals of length $|t-s|$. Each interval interacts with a bounded number of intervals in the stationary region. Rescale both curve variables by $|t-s|$. The kernel factor becomes $h^{-1}$, and the phase has parameter $|t-s|/h$ with mixed second derivative bounded away from zero. The one-dimensional oscillatory-integral estimate gives
\[
 Ch^{-1}(|t-s|/h)^{-1/2}
 =Ch^{-1/2}|t-s|^{-1/2}.
\]
Stationary phase gives uniform fixed-order bounds for the rescaled amplitudes. Its remainder has arbitrary powers of $h/|t-s|$ and obeys the same operator bound. Bounded overlap permits summation over the intervals.

On the complementary region $|r-u|\ge C|t-s|$, nonstationary phase gives
\[
 |[W_h(t)W_h(s)^*](r,u)|\le C_Nh^{-2}(1+|r-u|/h)^{-N}.
\]
Its Schur norm is at most $C_Nh^{N-2}|t-s|^{1-N}$, which is bounded by the right side of~\eqref{eq:local-2} for $N\ge2$ and $|t-s|\ge h$. The parametrix error satisfies these bounds as well. This proves~\eqref{eq:local-1}--\eqref{eq:local-2}.

Interpolating~\eqref{eq:local-1} and~\eqref{eq:local-2}, with weight $2/3$ on the second, gives
\[
 \norm{W_h(t)W_h(s)^*}{L^{3/2}(I)\to L^3(I)}
 \lesssim h^{-2/3}(h+|t-s|)^{-2/3}.
\]
Since $(h+|t-s|)^{-2/3}\le|t-s|^{-2/3}$, the one-dimensional Hardy--Littlewood--Sobolev inequality and $TT^*$ imply
\[
 \norm{W_hf}{L^3_{t,r}}\lesssim h^{-1/3}\norm f2.
\]
Returning to physical time multiplies the $L^3$ norm by $h^{1/3}$, which proves~\eqref{eq:local-main} on an interval of length $t_1h$. A fixed number of intervals handles each fixed $A$. Time translation is absorbed by unitarity.
\end{proof}

For the semiclassical evolution above, the characteristic equation is
\[
 |\xi|_g^2-h\tau=0,
\]
where $\tau$ is the Fourier variable in~\eqref{eq:FT-convention}. In the semiclassical covariables $(h\tau,\xi)$, the graph has definite second fundamental form. Thus Tacy's spacetime and observation dimensions are $3$ and $2$.

\subsection{Convolution in time}\label{subsec:time-convolution}
Let $K(t-s)$ be an operator-valued convolution kernel on $L^2(I)$. Write $\mathcal K$ for the associated operator:
\[
 (\mathcal KF)(t)=\int_\R K(t-s)F(s)\dd s.
\]
Plancherel gives
\begin{equation}\label{eq:operator-Plancherel}
 \norm{\mathcal K}{L^2(\R\times I)\to L^2(\R\times I)}
 \le\operatorname*{ess\,sup}_{\tau\in\R}
 \norm{\widehat K(\tau)}{L^2(I)\to L^2(I)}.
\end{equation}
This follows first on compactly supported smooth functions and then on $L^2$ by density.

We will also use the following observation. If $A(t,r,s,u)$ is the kernel of a bounded operator $\mathcal A:L^p\to L^q$ and $\chi\in C_c^\infty(\R)$, multiplication of this kernel by $\chi((t-s)/L)$ preserves the operator bound uniformly in $L>0$. In fact, with $\chi_L(t)=\chi(t/L)$, the new operator is
\begin{equation}\label{eq:modulation}
 \frac1{2\pi}\int_\R\widehat\chi_L(\tau)
 e^{-it\tau}\mathcal A e^{is\tau}\dd\tau,
 \qquad
 \norm{\widehat\chi_L}{L^1}=\norm{\widehat\chi}{L^1}.
\end{equation}
The claim follows since input and output modulations are isometries.

We also use the following block estimate. Let $\{D_j\}$ be disjoint measurable cells. Suppose a block operator $\mathcal T$ has norm at most $A$ from $L^{3/2}(D_j)$ to $L^3(D_k)$, and each input or output cell is coupled to at most $D$ cells. Then
\[
 \norm{(\mathcal T F)|_{D_k}}3
 \le A\sum_{j\sim k}\norm{F|_{D_j}}{3/2}.
\]
The adjacency operator is bounded on $\ell^3$, while $\ell^{3/2}\subset\ell^3$. Consequently
\begin{equation}\label{eq:block-lemma}
 \norm{\mathcal T F}3\le C_DA\norm F{3/2}.
\end{equation}
The same argument applies to partitions of bounded overlap.

Fix an even $\vartheta_0\in C_c^\infty((-2,2))$ equal to one on $[-1,1]$.

\subsection{The Schr\"odinger kernel}
For the remainder of this section, $t,s$ denote rescaled time, with physical time equal to $ht$. Put $a_0=|\beta|^2$ and write
\[
 K_h(t-s,r,u)=\bigl(e^{i(t-s)h\Delta_g}a_0(hP)\bigr)(\gamma(r),\gamma(u)).
\]
Fix a sufficiently large constant $t_0$. We use the following kernel expansion.

\begin{lemma}\label{lem:kernel-reduction}
There are fixed constants $C,C_0$ such that, for $\lambda\ge2$, $0<c\le1$, and
\[
 t_0\le|t-s|\le2c\log\lambda,
\]
\begin{equation}\label{eq:kernel-reduction}
 K_h(t-s,r,u)=\sum_{\alpha\in\mathcal G}
 \frac{w_\alpha(r,u)}{4\pi ih(t-s)}
 e^{id_\alpha(r,u)^2/(4h(t-s))}
 a_0\left(\frac{d_\alpha(r,u)}{2|t-s|}\right)
 +O\bigl(\lambda^{1/2+C_0c}\bigr),
\end{equation}
uniformly in $t,s,r,u$. The principal sum is finite, and each nonzero term has $|t-s|<d_\alpha<4|t-s|$.
\end{lemma}
\begin{proof}
We use the functional-calculus argument of~\cite{HSTZ}.

Extend $a_0$ evenly to $\R$, and define
\[
 m_{h,t,s}(\nu)=e^{-i(t-s)h\nu^2}a_0(h\nu),\qquad
 \widehat m_{h,t,s}(v)=\int e^{iv\nu}m_{h,t,s}(\nu)\dd\nu.
\]
Since $m_{h,t,s}$ is even,
\begin{align}
 \widehat m_{h,t,s}(v)
 &=h^{-1}\int e^{-i((t-s)\rho^2-v\rho)/h}a_0(\rho)\dd\rho,\label{eq:wave-multiplier}\\
 m_{h,t,s}(P)
 &=\frac1{2\pi}\int_\R\widehat m_{h,t,s}(v)\cos(vP)\dd v.\label{eq:cosine-functional}
\end{align}
Insert an even smooth cutoff in $v/(c\log\lambda)$ which equals one on $[-16,16]$ and vanishes outside $(-32,32)$. On the discarded part,
\[
 |2(t-s)\rho-v|\ge |v|/2
 \quad\text{for}\quad \rho\in\supp a_0,
 \quad |t-s|\le2c\log\lambda.
\]
Repeated integration by parts in $\rho$ gives, for fixed integers $N,j\ge0$,
\begin{equation}\label{eq:wave-tail}
 |\partial_v^j\widehat m_{h,t,s}(v)|
 \le C_{N,j}h^{N-j-1}(1+|v|)^{-N},
 \qquad |v|\ge16c\log\lambda.
\end{equation}
Integration by parts in $v$ makes the discarded even multiplier $O(h^N)$ in every fixed Schwartz seminorm. These bounds are uniform because $c\log\lambda\ge t_0/2$ in the stated time range. In particular, its graph norm from $H^{-2}(M)$ to $H^2(M)$ is $O(h^N)$. Uniform point evaluation on $H^2(M)$ gives the same bound for its scalar kernel.

For the remaining wave times, use the Hadamard parametrix on $(\widetilde M,\widetilde g)$:
\begin{equation}\label{eq:wave-parametrix}
 (\cos(v\sqrt{-\Delta_{\widetilde g}}))(x,y)
 =\sum_{m=0}^{N_0}w_m(x,y)W_m(v,x,y)+R_{N_0}(v,x,y),
\end{equation}
where the Hadamard coefficients $w_m$ are normalized with the radial distributions below, and
\begin{equation}\label{eq:W0-HSTZ}
 W_0(v,x,y)=(2\pi)^{-2}\int_{\R^2}
 e^{id_{\widetilde g}(x,y)\xi_1}\cos(v|\xi|)\dd\xi.
\end{equation}
For the higher terms we use the normalization in~\cite[the paragraph preceding (2.34)]{SZ}. The radial distributions satisfy the recursion $2W_m(v)=v\int_0^vW_{m-1}(w)\dd w$, with the spatial variables fixed. Equivalently,
\[
 W_m(v,x,y)=(2\pi)^{-2}\int_{\R^2}
 e^{id_{\widetilde g}(x,y)\xi_1}
 \left.(-\partial_z)^m\cos(v\sqrt z)\right|_{z=|\xi|^2}\dd\xi.
\]
Indeed, $-2\partial_z\cos(v\sqrt z)=v\int_0^v\cos(w\sqrt z)\dd w$, and repeated differentiation gives the recursion. These are distributional Fourier integrals. In particular, the first higher multiplier is $v\sin(v|\xi|)/(2|\xi|)$. They satisfy
\[
 (\partial_v^2-\Delta_{\R^2})W_m=m W_{m-1},\qquad
 \partial_d W_m=-\frac d2 W_{m-1},\quad m\ge1,
\]
where $d=d_{\widetilde g}(x,y)$. The Hadamard coefficients satisfy
\[
 \left(m+d\partial_d+\frac d2\left(\Delta_{\widetilde g,x}d-\frac1d\right)\right)w_m
 =\Delta_{\widetilde g,x}w_{m-1},\quad m\ge1,
\]
with the corresponding homogeneous equation for $w_0$ and $w_0(x,x)=1$.
For each prescribed number of derivatives, choosing $N_0$ sufficiently large gives
\begin{equation}\label{eq:HSTZ-Hadamard-bounds}
 |\partial_v^jR_{N_0}(v,x,y)|\le Ce^{C|v|},
\end{equation}
uniformly in $x,y\in\widetilde M$. Along each connector, the transport equation is equivalently
\[
 \partial_d\bigl(d^m w_m/w_0\bigr)
 =d^{m-1}w_0^{-1}\Delta_{\widetilde g,x}w_{m-1}.
\]
Integrate from the diagonal. Jacobi comparison bounds $w_0$ and $w_0^{-1}$ exponentially, and the endpoint derivative bounds above give, inductively, $Ce^{Cd_{\widetilde g}(x,y)}$ for every fixed derivative of $w_m$. The corresponding fixed derivatives of distance have the same bounds away from the diagonal. Near the diagonal the transport coefficients are smooth with uniform bounds. See also~\cite[Appendix~B]{Keeler}.

The remainder in~\eqref{eq:wave-parametrix} has zero Cauchy data and satisfies
\[
 (\partial_v^2-\Delta_{\widetilde g,x})R_{N_0}=F_{N_0},
 \qquad F_{N_0}(v,x,y)
 =(\Delta_{\widetilde g,x}w_{N_0})(x,y)W_{N_0}(v,x,y).
\]
For $v\ge0$ and $m\ge1$, the chosen radial normalization gives
\[
 W_m(v,x,y)=c_m v
 \bigl(v^2-d_{\widetilde g}(x,y)^2\bigr)_+^{m-3/2},
\]
where $c_m$ depends only on $m$. Fix the largest required wave-time derivative order $j$ and take $N_0\ge2j+6$. Differentiating the displayed radial formula shows that $\partial_v^iF_{N_0}$, $0\le i\le j$, belongs to $H^{j+2}$ in $x$ and is supported in $B_{\widetilde g}(y,v)$. Its time derivatives of order less than $j$ vanish at $v=0$ in this space. Near the diagonal this follows from smoothness of squared distance and rescaling $x$ by $v$ in normal coordinates. At the wave front the remaining power of $v^2-d_{\widetilde g}(x,y)^2$ is positive after all these derivatives.

The coefficient bounds and
$\operatorname{Vol}(B_{\widetilde g}(y,v))\le Ce^{Cv}$ give
\[
 \sup_y\|\partial_v^iF_{N_0}(v,\cdot,y)\|_{H^{j+2}(\widetilde M)}
 \le C_j e^{C_jv},\qquad 0\le i\le j.
\]
The vanishing initial jets allow time differentiation in Duhamel's formula without boundary terms:
\[
 \partial_v^jR_{N_0}(v,\cdot,y)
 =\int_0^v\frac{\sin((v-w)\sqrt{-\Delta_{\widetilde g}})}
 {\sqrt{-\Delta_{\widetilde g}}}
 \partial_w^jF_{N_0}(w,\cdot,y)\dd w.
\]
The sine multiplier has norm at most $v-w$ on each graph Sobolev space. Consequently
\[
 \sup_y\|\partial_v^jR_{N_0}(v,\cdot,y)\|_{H^{j+2}}
 \le C_j\int_0^v(v-w)e^{C_jw}\dd w
 \le C_j'e^{C_j'v}.
\]
Uniform Sobolev evaluation proves~\eqref{eq:HSTZ-Hadamard-bounds}. Fixed $y$-derivatives commute with the wave equation in $x$, so increasing $N_0$ gives the corresponding mixed derivative bounds. Evenness handles negative $v$. This applies the wave-energy argument of~\cite[Section~3]{Keeler} on the cover, before periodization.

The leading coefficient $w_0$ is the inverse square root of the normal-coordinate volume density. See~\cite[proof of (2.36)]{SZ} and~\cite{HSTZ}. Geodesic polar coordinates centered at $x$ have volume element
\[
 J_{x,\theta}(r)\,dr\,d\theta,
\]
where $J_{x,\theta}(0)=0$ and $\partial_rJ_{x,\theta}(0)=1$, whereas Euclidean polar measure in $T_x\widetilde M$ is $r\,dr\,d\theta$. Hence the normal-coordinate density is $J_{x,\theta}(r)/r$. Writing $J_{x,y}$ for the normalized scalar normal Jacobi solution along the geodesic from $x$ to $y$, we therefore have
\begin{equation}\label{eq:w0-normalization}
 w_0(x,y)=\left(\frac{d_{\widetilde g}(x,y)}{J_{x,y}(d_{\widetilde g}(x,y))}\right)^{1/2}.
\end{equation}
Thus $w_0(\widetilde\gamma(r),\alpha\widetilde\gamma(u))=w_\alpha(r,u)$.

Finite wave propagation and~\eqref{eq:orbit-count} now give one finite set
\[
 \mathcal G_{\lambda,c}
 =\{\alpha:d_{\widetilde g}(x_0,\alpha x_0)\le32c\log\lambda+4\ell\},
 \qquad \#\mathcal G_{\lambda,c}\le C\lambda^{Cc},
\]
which contains all effective deck transformations for endpoints in the observed rectangle. At the retained wave times, periodization gives
\[
 (\cos(vP))(\gamma(r),\gamma(u))
 =\sum_{\alpha\in\mathcal G_{\lambda,c}}
 (\cos(v\sqrt{-\Delta_{\widetilde g}}))(\widetilde\gamma(r),\alpha\widetilde\gamma(u)).
\]
The wave-time cutoff can be removed separately from each of the finitely many Hadamard terms. The discarded even multiplier has every fixed Schwartz seminorm $O(h^N)$ by~\eqref{eq:wave-tail}. The derivatives in $\nu^2$ satisfy the same bounds, including at $\nu=0$ by evenness. Each radial extension has $L^1(\R^2)$ norm $O(h^N)$, which bounds the scalar-kernel error. The coefficient bounds and the finite deck sum multiply this error by at most $\lambda^{Cc}$.

For the leading term, evenness and~\eqref{eq:cosine-functional} give
\[
 \frac1{2\pi}\int_{\R}\widehat m_{h,t,s}(v)\cos(v|\xi|)\dd v
 =m_{h,t,s}(|\xi|).
\]
Thus pairing~\eqref{eq:W0-HSTZ} with~\eqref{eq:wave-multiplier}, and removing the wave-time cutoff by~\eqref{eq:wave-tail}, gives for each $\alpha\in\mathcal G_{\lambda,c}$
\[
 \frac{w_\alpha}{(2\pi h)^2}\int_{\R^2}
 e^{i(d_\alpha\xi_1-(t-s)|\xi|^2)/h}a_0(|\xi|)\dd\xi.
\]
By quadratic stationary phase, uniformly in the critical point $(d_\alpha/(2(t-s)),0)$, this equals
\begin{equation}\label{eq:quadratic-leading}
 \frac{w_\alpha}{4\pi ih(t-s)}
 e^{id_\alpha^2/(4h(t-s))}a_0(d_\alpha/(2|t-s|))
 +O(w_\alpha|t-s|^{-2}).
\end{equation}
Completing the square and Fourier inverting $a_0(|\xi|)$ bounds the error inside the stationary-phase bracket by
\[
 C\frac h{|t-s|}\int_{\R^2}|\xi|^2
 \big|\widehat{a_0(|\cdot|)}(\xi)\big|\dd\xi.
\]
The same formula covers critical points outside the frequency support, where the principal amplitude vanishes. The stationary-phase errors sum to $O(\lambda^{Cc})$. A nonzero principal amplitude satisfies $|t-s|<d_\alpha<4|t-s|$. Thus the principal sum can be extended from $\mathcal G_{\lambda,c}$ to $\mathcal G$.

For a higher term, Fourier inversion gives the radial multiplier
\[
 \left.(-\partial_z)^m
 \bigl(e^{-i(t-s)hz}a_0(h\sqrt z)\bigr)\right|_{z=|\xi|^2}.
\]
After rescaling the Fourier variable by $h^{-1}$, this is
$h^m e^{-i(t-s)|\xi|^2/h}b_{m,h,t,s}(\xi)$, where
\[
 b_{m,h,t,s}(\xi)
 =\sum_{j=0}^{m}\binom{m}{j}
 \bigl(i(t-s)\bigr)^{m-j}(-h)^j
 \left.\partial_z^j a_0(\sqrt z)\right|_{z=|\xi|^2}.
\]
These amplitudes have fixed compact frequency support and satisfy
$\|b_{m,h,t,s}\|_{C^j}\le C_{m,j}(1+|t-s|)^m$ for every fixed nonnegative integer $j$. Quadratic stationary phase, uniformly in $d=d_{\widetilde g}(x,y)$, therefore bounds the corresponding kernel by
\[
 C_m |w_m(x,y)|h^{m-1}|t-s|^{-1}(1+|t-s|)^m.
\]
The stationary-phase error is bounded by
$C_m|w_m|h^m|t-s|^{-2}(1+|t-s|)^m$ and satisfies the same upper bound. Since $m\ge1$, $t_0\le|t-s|\le2c\log\lambda$, and $d\le32c\log\lambda+6\ell$, the coefficient and polynomial time factors are bounded by $C\lambda^{Cc}$. Summing over the finite expansion and $\mathcal G_{\lambda,c}$ gives $O(\lambda^{Cc})$.

For the smooth remainder, substitute~\eqref{eq:wave-multiplier} and integrate by parts in $v$, as in~\cite{HSTZ}. Since $|\rho|$ is bounded below on $\supp a_0$, each integration gains a factor $h$, while the derivatives of the remainder and the wave-time cutoff have exponential bounds. For each fixed $N$, taking $N+1$ integrations by parts accounts for the initial factor $h^{-1}$ and bounds its contribution by
\[
 O\bigl(h^N e^{C_Nc\log\lambda}\bigr)
 =O\bigl(\lambda^{-N+C_Nc}\bigr)
\]
after choosing the expansion order sufficiently large. The finite deck sum again changes only $C_N$. For a fixed sufficiently large $N$, this is bounded by $C\lambda^{1/2+C_0c}$.
Together these estimates prove~\eqref{eq:kernel-reduction}.
\end{proof}

\subsection{Fourier transform in time}
With $h=\lambda^{-1}$ and $\psi=\beta$ in Lemma~\ref{lem:local}, write
\[
 W_h f(t,r)=\bigl(e^{ith\Delta_g}\beta(hP)f\bigr)(\gamma(r)).
\]
Since physical time equals $ht$ after this rescaling, the required physical-time estimate is equivalent to
\begin{equation}\label{eq:log-target}
 \norm{W_h}{L^2(M)\to L^3([0,c\log\lambda]\times I)}\lesssim h^{-1/3},
 \qquad
 \norm{W_h W_h^*}{L^{3/2}([0,c\log\lambda]\times I)\to L^3([0,c\log\lambda]\times I)}\lesssim h^{-2/3}.
\end{equation}
We prove the second inequality with $c>0$ chosen sufficiently small at the end.

The contribution of $\vartheta_0((t-s)/t_0)K_h(t-s,r,u)$ has norm $O(h^{-2/3})$. Indeed, divide semiclassical time into fixed-length intervals. Lemma~\ref{lem:local} bounds each restricted evolution by $Ch^{-1/3}$, so a $TT^*$ block has norm at most $Ch^{-2/3}$. The cutoff in time difference leaves boundedly many interactions and preserves their norms by~\eqref{eq:modulation}. Equation~\eqref{eq:block-lemma} gives the bound independently of $\lambda$.

Set
\[
 \vartheta_\lambda(v)=\bigl(1-\vartheta_0(v/t_0)\bigr)
 \vartheta_0\bigl(v/(c\log\lambda)\bigr).
\]
Multiply each principal term of~\eqref{eq:kernel-reduction} by $\vartheta_\lambda$ and denote the resulting scalar kernel by $K_{h,\alpha}(t-s,r,u)$. We estimate these convolution operators on the whole time line, and then restrict both time variables to $[0,c\log\lambda]$.

For fixed $\alpha,r,u$, write $d=d_\alpha(r,u)>0$, $w=w_\alpha(r,u)$, and substitute
\begin{equation}\label{eq:time-dv}
 t-s=dv
\end{equation}
in the time Fourier transform. This gives
\begin{equation}\label{eq:time-phase}
 \widehat K_{h,\alpha}(\tau,r,u)
 =\frac{w}{4\pi ih}\int
 e^{i(d/h)(1/(4v)+h\tau v)}
 \frac1v a_0\left(\frac1{2|v|}\right)\vartheta_\lambda(dv)\dd v.
\end{equation}
The amplitude is supported in $1/4\le|v|\le1$. Its fixed-order derivatives are uniformly bounded. The logarithmic cutoff gives factors $d/(c\log\lambda)$, which are bounded on the support. Derivatives of the short-time cutoff occur only when $d\asymp t_0$.

For every $\tau\in\R$,
\[
 \partial_v^2\left(\frac1{4v}+h\tau v\right)=\frac1{2v^3}.
\]
One-dimensional van der Corput on the positive and negative $v$-intervals therefore gives
\begin{equation}\label{eq:Fourier-coarse}
 |\widehat K_{h,\alpha}(\tau,r,u)|
 \le Ch^{-1/2}w_\alpha d_\alpha^{-1/2}
 =Ch^{-1/2}J_\alpha(d_\alpha)^{-1/2}.
\end{equation}
The estimate is uniform over the whole temporal frequency line.

For the terms outside the tube we use the oscillatory refinement of~\eqref{eq:Fourier-coarse}. All possible stationary values of $h\tau$ lie in a fixed positive compact interval. In a slightly larger such interval write $h\tau=\rho^2$, where $\rho\asymp1$. The critical points and critical values of the phase in~\eqref{eq:time-phase} are
\[
 v_\pm=\pm(2\rho)^{-1},\qquad
 (4v_\pm)^{-1}+\rho^2v_\pm=\pm\rho.
\]
Uniform stationary phase gives, for fixed nonzero constants $c_\pm$,
\begin{equation}\label{eq:Fourier-osc}
 \widehat K_{h,\alpha}(\tau,r,u)
 =h^{-1/2}\sum_\pm e^{\pm i\rho d_\alpha(r,u)/h}
 c_\pm J_\alpha(d_\alpha)^{-1/2}\rho^{-1/2}a_0(\rho)
 \vartheta_\lambda\left(\frac{\pm d_\alpha}{2\rho}\right)
 +O\bigl(h^{1/2}w_\alpha d_\alpha^{-3/2}\bigr).
\end{equation}
The $C^1$ norm in $(r,u)$ of each displayed principal amplitude is at most $C\lambda^{Cc}$.
Smooth compact support makes the expansion uniform as a critical point crosses the support boundary. Outside the enlarged stationary frequency interval, the derivative of the reduced phase is bounded away from zero. One integration by parts gives
\begin{equation}\label{eq:Fourier-nonstat}
 |\widehat K_{h,\alpha}(\tau,r,u)|\le Cw_\alpha d_\alpha^{-1}.
\end{equation}
We estimate the remainder by Schur's test and apply the oscillatory estimate to the displayed principal amplitudes.

The spatial $L^2$ bounds for the phases $\pm\rho d/h$ give spacetime $L^2$ bounds by~\eqref{eq:operator-Plancherel}.

\subsection{Tube and oscillatory terms}
\paragraph{Tube terms.} Let $\mathcal K_h^{\mathrm{tube}}$ denote the spacetime convolution operator defined by the sum over $\mathcal G_{\mathrm{tube}}$ in~\eqref{eq:tube-set}. On the principal support, $d_\alpha\asymp|t-s|$. By~\eqref{eq:tube-count}, at most $C(1+|t-s|)$ terms are present, and each has size at most $C(h|t-s|)^{-1}$. Since $|t-s|\ge t_0$,
\begin{equation}\label{eq:sch-tube-one}
 \norm{\mathcal K_h^{\mathrm{tube}}}{L^1_{t,r}\to L^\infty_{t,r}}\le Ch^{-1}.
\end{equation}
On the Fourier side,~\eqref{eq:Jacobi-comparison} and~\eqref{eq:Fourier-coarse} give
\[
 |\widehat K_{h,\alpha}(\tau,r,u)|\le Ch^{-1/2}e^{-d_\alpha/C}
\]
for the separated distances under consideration. The difference between $d_\alpha(r,u)$ and $d_{\widetilde g}(x_0,\alpha x_0)$ is at most $2\ell$. Decomposing the latter distance dyadically and using~\eqref{eq:tube-count}, we obtain
\begin{equation}\label{eq:sch-tube-sum}
 \sup_{\tau,r,u}\sum_{\alpha\in\mathcal G_{\mathrm{tube}}}
 |\widehat K_{h,\alpha}(\tau,r,u)|
 \le Ch^{-1/2}\left(1+\sum_{j\ge0}2^je^{-2^j/C}\right)
 \le Ch^{-1/2}.
\end{equation}
The curve interval has fixed length. Schur's test in the curve variables and Plancherel in time imply
\begin{equation}\label{eq:sch-tube-two}
 \norm{\mathcal K_h^{\mathrm{tube}}}{L^2_{t,r}\to L^2_{t,r}}\le Ch^{-1/2}.
\end{equation}
Interpolation of~\eqref{eq:sch-tube-one} and~\eqref{eq:sch-tube-two} gives
\begin{equation}\label{eq:sch-tube-critical}
 \norm{\mathcal K_h^{\mathrm{tube}}}{L^{3/2}_{t,r}\to L^3_{t,r}}
 \lesssim (h^{-1})^{1/3}(h^{-1/2})^{2/3}=h^{-2/3}.
\end{equation}
\paragraph{Oscillatory terms.} Let $\mathcal K_h^{\mathrm{osc}}$ denote the spacetime convolution operator defined by the sum over $\mathcal G_{\mathrm{osc}}$ in~\eqref{eq:tube-set}. For an effective transformation, the support of the principal kernel gives $t_0<d_\alpha(r_0,u_0)<8c\log\lambda$ at some point of the parameter rectangle. Since $d_\alpha$ varies by at most $2\ell$ on that rectangle, fixing $t_0>3+2\ell$ gives
\[
 3<d_\alpha(r,u)<8c\log\lambda+2\ell.
\]
Apply~\eqref{eq:tube-transversality} with $D=8c\log\lambda+2\ell$. It gives the lower bound $C^{-1}\lambda^{-Cc}$ for the mixed-derivative alternative. Take the spatial adjoint in~\eqref{eq:Fourier-osc}. The repeated derivative is then in the input variable, as required in~\eqref{eq:Blair}. The phase parameter is $\rho/h\asymp h^{-1}$, and the amplitude has $C^1$ norm at most $C\lambda^{Cc}$. Applying~\eqref{eq:Blair} and summing by~\eqref{eq:orbit-count} bounds the principal operators by
\[
 Ch^{-1/2}h^{1/4}\lambda^{Cc}=Ch^{-1/4}\lambda^{Cc}.
\]
For the remainders, use Schur's test. Since $d_\alpha\ge3$ and the curve interval has fixed length,~\eqref{eq:Fourier-osc} and Schur give $O(h^{1/2}\lambda^{Cc})$ after summing the at most $C\lambda^{Cc}$ effective transformations. Outside the stationary temporal-frequency region,~\eqref{eq:Fourier-nonstat}, $w_\alpha\le1$, and $d_\alpha\ge3$ give $O(\lambda^{Cc})$. Both are bounded by $Ch^{-1/4}\lambda^{Cc}$ for $0<h\le1$. After increasing $C_0$ in Lemma~\ref{lem:kernel-reduction}, these bounds and Plancherel give
\begin{equation}\label{eq:sch-osc-two}
 \norm{\mathcal K_h^{\mathrm{osc}}}{L^2_{t,r}\to L^2_{t,r}}
 \le Ch^{-1/4}\lambda^{C_0c}.
\end{equation}
The original kernel and~\eqref{eq:orbit-count} also give
\begin{equation}\label{eq:sch-osc-one}
 \norm{\mathcal K_h^{\mathrm{osc}}}{L^1_{t,r}\to L^\infty_{t,r}}
 \le Ch^{-1}\lambda^{C_0c}.
\end{equation}
Consequently,
\begin{equation}\label{eq:sch-osc-critical}
 \norm{\mathcal K_h^{\mathrm{osc}}}{L^{3/2}_{t,r}\to L^3_{t,r}}
 \le C(h^{-1})^{1/3}(h^{-1/4})^{2/3}\lambda^{C_0c}
 =Ch^{-1/2}\lambda^{C_0c}.
\end{equation}
The gain over $h^{-2/3}$ allows us to absorb $\lambda^{C_0c}$ by choosing $c$ small.

Restrict the preceding convolution bounds to $[0,c\log\lambda]$. The error kernel in Lemma~\ref{lem:kernel-reduction}, estimated directly on $([0,c\log\lambda]\times I)^2$, has $L^{3/2}\to L^3$ norm at most
\begin{equation}\label{eq:error-norm}
 C\lambda^{1/2+C_0c}(\log\lambda)^{2/3}.
\end{equation}
Indeed a kernel bounded by $B$ on a finite measure set $D\times D$ has this operator norm at most $B|D|^{2/3}$.

Combining the short-time part,~\eqref{eq:sch-tube-critical},~\eqref{eq:sch-osc-critical}, and~\eqref{eq:error-norm}, we obtain
\begin{equation}\label{eq:log-final-TT}
 \norm{W_h W_h^*}{3/2\to3}
 \le Ch^{-2/3}+Ch^{-1/2}\lambda^{C_0c}
 \bigl(1+(\log\lambda)^{2/3}\bigr).
\end{equation}
With the derivative and expansion orders fixed, choose $c>0$ so that $C_0c<1/12$. Since $h=\lambda^{-1}$, the ratio of the second term in~\eqref{eq:log-final-TT} to $h^{-2/3}$ is
\[
 C\lambda^{-1/6+C_0c}\bigl(1+(\log\lambda)^{2/3}\bigr),
\]
which tends to zero. This proves~\eqref{eq:log-target}. Taking the $TT^*$ square root and multiplying the norm by $\lambda^{-1/3}$ proves~\eqref{eq:main-log}. For a longer segment, subdivide it into finitely many short segments and sum the cubes of the output norms. Corollary~\ref{cor:closed} follows by applying~\eqref{eq:main-log} on $O(\lambda/(\log\lambda))$ physical-time intervals and summing their third powers.\qed

\section{All-time Schr\"odinger restriction}\label{sec:gluing}
\subsection{The Cartan--Hadamard background estimate}\label{sec:background}
We first prove the complete-geodesic estimate on a Cartan--Hadamard background.

\begin{proposition}[Cartan--Hadamard background estimate]\label{prop:background}
Let $(Y,g_0)$ be a smooth complete asymptotically hyperbolic surface of smooth bounded geometry. Suppose that $Y$ is simply connected and $K_{g_0}\le0$. Use $P_0,U_0,Q_0$ and $S(\kappa)$ as defined in Section~\ref{subsec:background-spectral}. For every complete unit-speed geodesic $\Gamma:\R\to Y$ and every fixed $\psi\in C_c^\infty((0,\infty))$, there exist $\lambda_0\ge1$ and $C_\psi>0$ such that
\begin{equation}\label{eq:background-high}
 \norm{R_\Gamma U_0(t)\psi(P_0/\lambda)\Pczer f}{L^3(\R\times\Gamma)}
 \le C_\psi\norm f{L^2(Y)},\qquad \lambda\ge\lambda_0.
\end{equation}
If the threshold $1/4$ is regular, then
\begin{equation}\label{eq:background-full}
 \norm{R_\Gamma U_0(t)\Pczer f}{L^3(\R\times\Gamma)}
 \le C\norm f{L^2(Y)}.
\end{equation}
More precisely, for $0<\varepsilon\le1$,
\begin{equation}\label{eq:background-low}
 \norm{R_\Gamma U_0(t)\1_{(0,\varepsilon]}(Q_0)\Pczer f}{L^3(\R\times\Gamma)}
 \le C\varepsilon^{5/6}\norm f{L^2(Y)}.
\end{equation}
The constants are uniform over complete geodesics on the fixed background. If there is no point spectrum, $\Pczer$ can be omitted in~\eqref{eq:background-full}.
\end{proposition}

In this subsection, $(Y,g_0)$ and $\Gamma$ satisfy the hypotheses of Proposition~\ref{prop:background}. For the high-frequency argument set $h=\lambda^{-1}$. We use the spectral-measure bounds from Section~\ref{subsec:background-spectral}.

\subsubsection{The high-frequency \texorpdfstring{$TT^*$}{TT*} estimate}
Let $\psi\in C_c^\infty((0,\infty))$ and set
\[
 V_h f(t,r)=e^{it/4}\bigl[R_\Gamma U_0(t)\psi(hP_0)\Pczer f\bigr](r),\qquad
 b_h(\kappa)=\left|\psi\!\left(h\sqrt{\kappa^2+1/4}\right)\right|^2.
\]
On $\supp b_h$ one has $\kappa\asymp h^{-1}$, and $|\partial_\kappa^jb_h|\le C_jh^j$. The factor $e^{it/4}$ preserves the required norm and removes the spectral shift. The kernel of $V_hV_h^*$ is
\begin{equation}\label{eq:asymp-K}
 K_{0,h}(t-s,r,u)=\int_0^\infty e^{-i(t-s)\kappa^2}b_h(\kappa)
 S(\kappa,\Gamma(r),\Gamma(u))\dd \kappa.
\end{equation}
Here $t$ and $s$ are the two physical time variables. All identities are initially tested on smooth compactly supported functions of $(t,r)$.

Taking the time Fourier transform of~\eqref{eq:asymp-K} gives
\begin{equation}\label{eq:asymp-Jacobian}
 \widehat K_{0,h}(\tau)=2\pi\int_0^\infty\delta(\tau-\kappa^2)b_h(\kappa)
 R_\Gamma S(\kappa)R_\Gamma^*\dd \kappa
 =\frac{\pi}{\kappa}b_h(\kappa)R_\Gamma S(\kappa)R_\Gamma^*,\qquad \kappa=\sqrt\tau>0.
\end{equation}
The multiplier is zero for $\tau\le0$. Let $\mathcal K_{0,h}$ be the spacetime convolution operator with scalar kernel $K_{0,h}$. Using~\eqref{eq:CH-Schur} and Plancherel, we obtain
\begin{equation}\label{eq:asymp-L2}
 \norm{\mathcal K_{0,h}}{L^2(\R^2)\to L^2(\R^2)}\le Ch^{1/2}.
\end{equation}

Choose a small fixed $\delta_0>0$, to be specified below. Using the cutoff $\vartheta_0$ from Section~\ref{subsec:time-convolution}, write $K_{0,h}^{\mathrm{short}}(t-s)=\vartheta_0((t-s)/(\delta_0h))K_{0,h}(t-s)$ and $K_{0,h}^{\mathrm{long}}=K_{0,h}-K_{0,h}^{\mathrm{short}}$, and denote their associated spacetime convolution operators by $\mathcal K_{0,h}^{\mathrm{short}}$ and $\mathcal K_{0,h}^{\mathrm{long}}$.
The modulation formula~\eqref{eq:modulation} preserves the bound~\eqref{eq:asymp-L2} for both pieces. In particular,
\begin{equation}\label{eq:asymp-long-two}
 \norm{\mathcal K_{0,h}^{\mathrm{long}}}{2\to2}\le Ch^{1/2}.
\end{equation}

We next prove
\begin{equation}\label{eq:asymp-long-one}
 \sup_{|t-s|\ge \delta_0h}\sup_{r,u}|K_{0,h}(t-s,r,u)|\le C_{\delta_0} h^{-1}.
\end{equation}
Put $d=d_{g_0}(x,y)$. Choose a fixed $d_0>0$ sufficiently small relative to $\delta_0$ and the frequency annulus, and then take $h$ sufficiently small. For $d\lesssim h$, use~\eqref{eq:CH-point} and~\eqref{eq:CH-whole-derivative}. On the band,
\[
 \norm{b_hS}{L^1_\kappa}\le Ch^{-2},\qquad
 \norm{\partial_\kappa(b_hS)}{L^1_\kappa}\le Ch^{-1}.
\]
Since $|\partial_\kappa(-(t-s)\kappa^2)|\gtrsim|t-s|/h$ and $b_h$ is smoothly supported, integration by parts gives
\begin{equation}\label{eq:asymp-ibp}
 \left|\int e^{-i(t-s)\kappa^2}b_h(\kappa)S(\kappa,x,y)\dd \kappa\right|
 \le \frac{Ch}{|t-s|}\norm{\partial_\kappa(b_hS)}1+
       \frac{Ch^2}{|t-s|}\norm{b_hS}1
 \le C|t-s|^{-1}\le C_{\delta_0} h^{-1}.
\end{equation}

For the remaining small-distance range $h\lesssim d<d_0$, use the separated phases
\[
 \phi_\pm(\kappa)=-(t-s)\kappa^2\pm \kappa d.
\]
The choice of $d_0$ gives $|\phi_\pm'|\gtrsim|t-s|/h$ for either sign of $t-s$. The symbol bounds~\eqref{eq:CH-symbol} give the same sufficient $L^1$ bounds for $b_h a_\pm$ and $\partial_\kappa(b_h a_\pm)$. Thus~\eqref{eq:asymp-ibp} applies to the principal terms. Here the derivative falls on $b_h a_\pm$, with the oscillation $e^{\pm i\kappa d}$ retained in the phase. The residual term is bounded by $C_Nh^{N-1}e^{-d/2}$ after integration in $\kappa$.

Finally, for $d\ge d_0$, the amplitude estimates give
\[
 \norm{b_h a_\pm}{L^\infty_\kappa}+
 \norm{\partial_\kappa(b_h a_\pm)}{L^1_\kappa}
 \le C_{d_0}h^{-1/2}e^{-d/C}.
\]
For $d_0\le d\le1$ the exponential factor can be absorbed into the constant, while for $d\ge1$ it is supplied by~\eqref{eq:CH-symbol}. Since $\phi_\pm''=-2(t-s)$, second-order van der Corput gives
\[
 \left|\int e^{i\phi_\pm(\kappa)}b_h(\kappa)a_\pm(\kappa,x,y)\dd \kappa\right|
 \le C_{d_0}h^{-1/2}|t-s|^{-1/2}e^{-d/C}\le C_{\delta_0} h^{-1}.
\]
The remainder is again estimated absolutely. This proves~\eqref{eq:asymp-long-one}.

Interpolation of~\eqref{eq:asymp-long-two} and~\eqref{eq:asymp-long-one} now gives
\begin{equation}\label{eq:asymp-long-critical}
 \norm{\mathcal K_{0,h}^{\mathrm{long}}}{L^{3/2}(\R^2)\to L^3(\R^2)}
 \le C(h^{-1})^{1/3}(h^{1/2})^{2/3}=C.
\end{equation}
\subsubsection{Short differences on the infinite observation line}
To control $K_{0,h}^{\mathrm{short}}$, we use both local propagation and the spatial decay in~\eqref{eq:CH-osc}. Partition $\Gamma$ into disjoint intervals $J_j$, $j\in\Z$, of a fixed sufficiently small length. Apart from a fixed finite number of neighboring intervals, two such intervals are separated by a fixed positive distance. Decrease $\delta_0$, and then $d_0$, if necessary so that on the support of the time-difference cutoff,
\[
 |2(t-s)\kappa|\le \frac12|r-u|
\]
whenever $r$ and $u$ belong to such a separated pair and $\kappa\in\supp b_h$. This is possible because $\kappa\asymp h^{-1}$ and is compatible with the local time constant in Lemma~\ref{lem:local}.

Partition time into $I_m=[m\delta_0h,(m+1)\delta_0h)$, $m\in\Z$. By Lemma~\ref{lem:local}, each block operator
\[
 \1_{I_m\times J_j}e^{it/4}R_\Gamma U_0(t)\psi(hP_0)\Pczer
\]
has uniformly bounded $L^2(Y)\to L^3(I_m\times J_j)$ norm. Hence the product of one block operator and the adjoint of another has uniformly bounded $L^{3/2}\to L^3$ norm. Its kernel is the corresponding block of~\eqref{eq:asymp-K} before the time-difference cutoff is inserted. The cutoff preserves this bound by~\eqref{eq:modulation}.

For neighboring spatial intervals, the support of the time-difference cutoff also leaves only a fixed number of neighboring time intervals. Each cell therefore has a bounded number of input and output neighbors. Apply~\eqref{eq:block-lemma} to sum these blocks.

For the remaining pairs, $d=|r-u|$ is bounded below by a fixed positive constant and $|t-s|\le2\delta_0h$. By the choice of $\delta_0$,
\[
 |\pm d-2(t-s)\kappa|\ge d/2.
\]
After rescaling $\kappa$ by $h^{-1}$, the phase has nonstationary parameter $d/h$. Its normalized derivatives are controlled since $|t-s|/(hd)$ is bounded. Repeated integration by parts, using~\eqref{eq:CH-symbol}, gives a bound $C_Nh^{N-3/2}d^{-N}e^{-d/C}$ for the principal terms. By increasing $N$ to attain any prescribed $M$ and using the fixed lower bound for $d$, we obtain, on these separated pairs,
\[
 |K_{0,h}^{\mathrm{short}}(t-s,r,u)|
 \le C_Mh^M\1_{\{|t-s|\le2\delta_0h\}}e^{-|r-u|/C}.
\]
The same bound follows for the spectral remainder by choosing $N$ sufficiently large in~\eqref{eq:CH-osc}. The right side has $L^{3/2}(\R^2)$ norm at most $C_Mh^{M+2/3}$. Young's inequality on $\R^2$ therefore controls the separated-pair part. Combining it with the neighboring-block estimate,
\begin{equation}\label{eq:asymp-short-critical}
 \norm{\mathcal K_{0,h}^{\mathrm{short}}}{3/2\to3}\le C.
\end{equation}

Combining~\eqref{eq:asymp-long-critical} and~\eqref{eq:asymp-short-critical}, for compact smooth tests $F$ we have
\[
 \norm{V_h^*F}2^2=\ip{\mathcal K_{0,h}F}{F}\le C\norm F{3/2}^2.
\]
Duality proves~\eqref{eq:background-high}, first on finite-band data and then by completion. The constants in the local propagation and spectral kernels are uniform over complete geodesics on this fixed background.

\subsubsection{Temporal synthesis and the threshold}
We record a summation argument for later use on $M$. For a fixed geodesic $\gamma$, assume $\|R_\gamma U(t)f_j\|_3\le C\|f_j\|_2$ uniformly for data supported in $P\asymp2^j$. Let $f$ be spectrally supported above a sufficiently large fixed frequency. Take a smooth dyadic partition of unity on this spectral half-line and write $f_j=\psi_j(P)f$. Choose the lower cutoff so that every annulus meeting the support of $f$ is covered by the single-band estimate. The Fourier transform of $R_\gamma U(t)f_j$ is supported in $\tau\asymp2^{2j}$. Temporal Littlewood--Paley synthesis and Minkowski's inequality give
\begin{align}
 \norm{\sum_jR_\gamma U(t)f_j}{L^3_{t,r}}
 &\le C\norm{\left(\sum_j|R_\gamma U(t)f_j|^2\right)^{1/2}}{L^3_{t,r}}\notag\\
 &\le C\left(\sum_j\norm{R_\gamma U(t)f_j}3^2\right)^{1/2}
 \le C\norm f2.\label{eq:temporal-synthesis}
\end{align}
For the first inequality, apply the scalar temporal square-function estimate at each $r$ and integrate in $r$. The second follows from Minkowski's inequality in $L^{3/2}_{t,r}$. Bounded overlap of the input spectral annuli gives $\sum_j\|f_j\|_2^2\lesssim\|f\|_2^2$. These estimates hold for finite sums, and the same bound on tails proves convergence. Applied to $P_0,U_0,\Gamma$, this sums~\eqref{eq:background-high} on $Y$.

Assume now that the threshold is regular. Equations~\eqref{eq:CH-low} and~\eqref{eq:line-distance} imply
\begin{equation}\label{eq:low-majorant}
 |S(\kappa,\Gamma(r),\Gamma(u))|
 \le C\kappa^2(1+|r-u|)e^{-|r-u|/2},
 \qquad0<\kappa\le1.
\end{equation}
The demodulated $TT^*$ kernel of
$R_\Gamma U_0(t)\1_{(0,\varepsilon]}(Q_0)\Pczer$ is
\[
 K_{0,\varepsilon}(t-s,r,u)=\int_0^\varepsilon e^{-i(t-s)\kappa^2}
 S(\kappa,\Gamma(r),\Gamma(u))\dd \kappa.
\]
Let $\mathcal K_{0,\varepsilon}$ denote the corresponding spacetime convolution operator. Absolute integration in $\kappa$ gives
\begin{equation}\label{eq:low-one}
 |K_{0,\varepsilon}(t-s,r,u)|\le C\varepsilon^3(1+|r-u|)e^{-|r-u|/2},\qquad
 \norm{\mathcal K_{0,\varepsilon}}{1\to\infty}\le C\varepsilon^3.
\end{equation}
On the other hand, Schur gives $\|R_\Gamma S(\kappa)R_\Gamma^*\|_{2\to2}\le C\kappa^2$. The time Fourier multiplier is $\pi \kappa^{-1}\1_{(0,\varepsilon]}(\kappa)R_\Gamma S(\kappa)R_\Gamma^*$, so
\begin{equation}\label{eq:low-two}
 \norm{\mathcal K_{0,\varepsilon}}{2\to2}\le C\varepsilon.
\end{equation}
Interpolation and $TT^*$ yield
\[
 \norm{\mathcal K_{0,\varepsilon}}{3/2\to3}
 \le C(\varepsilon^3)^{1/3}\varepsilon^{2/3}=C\varepsilon^{5/3},
 \qquad
 \norm{R_\Gamma U_0(t)\1_{(0,\varepsilon]}(Q_0)\Pczer}{2\to3}
 \le C\varepsilon^{5/6}.
\]
This proves~\eqref{eq:background-low}. The bounds are uniform for bounded smooth approximations of the spectral cutoff. The spectral theorem and the dual estimates then give the sharp-cutoff bound.

Under the present regular-threshold hypothesis,~\cite[Remark~1.4 and Theorem~1.3]{CH} gives, on every fixed compact interval in $(0,\infty)$, a uniform majorant $C(1+d)e^{-d/2}$ for the spectral kernel. Along $\Gamma$ this is integrable in $r-u$ by~\eqref{eq:line-distance}. The same $L^1\to L^\infty$ and temporal $L^2\to L^2$ argument therefore gives an $L^2\to L^3$ estimate on each such interval. A finite intermediate interval connects~\eqref{eq:background-low} to the high-frequency estimate obtained from~\eqref{eq:temporal-synthesis}. Combining these frequency ranges proves~\eqref{eq:background-full} and Proposition~\ref{prop:background}.\qed

\subsection{Local smoothing and retarded integrals}
For the remainder of this section, $(M,g)$ satisfies the hypotheses of Theorem~\ref{thm:asymp}. The rescaled boundary charts and the compact interior give smooth bounded geometry. Use the decomposition and auxiliary surfaces from Section~\ref{sec:AH-static}, together with~\eqref{eq:LS-interior}--\eqref{eq:derivative-LS}. The spatial commutators are supported in the lossless region $M_\infty$.

If the threshold is regular, then for every fixed bounded-energy cutoff $\Theta$ and compactly supported $\chi$,
\begin{equation}\label{eq:LS-low}
 \norm{\chi U(t)\Theta(P)\Pc f}{L^2(\R\times M)}\le C_{\chi,\Theta}\norm f2.
\end{equation}
This is the bounded-energy local smoothing estimate in~\cite[(2.45)]{HSTZ}.

Each chosen background has no eigenvalues and a regular threshold. Hence~\eqref{eq:background-full} applies without $\Pczer$, and background high-frequency smoothing is lossless for every compact cutoff.

We will also use the background local smoothing estimate
\begin{equation}\label{eq:background-full-smoothing}
 \norm{\chi U_0(t)f}{L^2(\R\times Y)}\le C_\chi\norm f2,
 \qquad \chi\in C_c^\infty(Y).
\end{equation}
We prove this estimate directly from the spectral kernels. On a fixed compact set,~\eqref{eq:CH-low} gives $\|\chi S(\kappa)\chi\|_{2\to2}\lesssim \kappa^2$ near zero. At high energy, integrating~\eqref{eq:CH-point} over a fixed compact set in two-dimensional polar coordinates gives $\|\chi S(\kappa)\chi\|_{2\to2}\lesssim_\chi \kappa^{1/2}$. After the change of variables $\tau=\kappa^2+1/4$, the time-Fourier multiplier of the localized $TT^*$ operator is $\pi \kappa^{-1}\chi S(\kappa)\chi$, so its norm is $O(\kappa)$ near zero and $O(\kappa^{-1/2})$ at high energy. The intermediate compact $\kappa$-range is controlled by the localized limiting absorption principle. This proves~\eqref{eq:background-full-smoothing}.

\begin{lemma}[Retarded restriction from local smoothing]\label{lem:CK-smoothing}
Let $(M,g)$ be a complete surface of smooth bounded geometry, with $U(t)=e^{it\Delta_g}$. Let $R_\gamma:H^2(M)\to L^3(\gamma)$ be a bounded geodesic restriction map, and let $J\subset\R$ be an interval. Let $B,\widetilde B$ be bounded spectral multipliers with $B\widetilde B=B$. Assume that $B$ has bounded spectral support and $\widetilde B$ is self-adjoint. For a real smooth bounded spatial cutoff $\chi$, suppose
\[
 \norm{R_\gamma U(t)Bf}{L^3(J\times\gamma)}\le A\norm f2,
 \qquad
 \norm{\chi U(t)\widetilde Bf}{L^2(J\times M)}\le D\norm f2.
\]
Then
\begin{equation}\label{eq:CK-smoothing}
 \left\|R_\gamma\int_{\{s\in J:s<t\}}U(t-s)B\chi F(s)\dd s
 \right\|_{L^3(J\times\gamma)}\le C AD\norm F{L^2(J\times M)}.
\end{equation}
The analogous bound holds with $s>t$ under the same assumptions.
\end{lemma}
\begin{proof}
Without the restriction $s<t$, the integral factors as
\[
 R_\gamma U(t)B\int_JU(-s)\widetilde B\chi F(s)\dd s.
\]
The outer factor has norm at most $A$. Since $\widetilde B$ is self-adjoint, the inner operator is the adjoint of $f\mapsto\chi U(s)\widetilde Bf$ and has norm at most $D$. Bounded spectral support of $B$ and the $H^2$ trace bound give an operator-valued time kernel from $L^2(M)$ to $L^3(\gamma)$. The Banach-valued Christ--Kiselev theorem~\cite{CK}, with time exponents $2<3$, gives~\eqref{eq:CK-smoothing}. Subtracting the retarded integral from the full integral gives the estimate for $s>t$.
\end{proof}

\subsection{Derivative estimates and frequency separation}
By~\eqref{eq:static-local-elliptic} and~\eqref{eq:derivative-LS}, every fixed differential operator $D$ of order $m$ with coefficients supported in a lossless transition region satisfies
\[
 \norm{DU(t)f_\lambda}{L^2(\R\times M)}
 \lesssim\lambda^{m-1/2}\norm{f_\lambda}2.
\]
The corresponding interior estimate has an additional factor $(\log\lambda)^{1/2}$. With the bounded-energy estimate \eqref{eq:LS-low} and local elliptic regularity, the same expression has bound $C\|f\|_2$ for $U(t)\Theta(P)\Pc f$ and fixed $D$.

The following form of the off-frequency argument in~\cite{HSTZ} applies to compact cutoffs and to cutoffs on a common end.

\begin{lemma}[Separated-frequency error]\label{lem:offband}
Let $(M,g)$ and $(Y,g_0)$ be complete surfaces of smooth bounded geometry, possibly the same surface, and set $P=\sqrt{-\Delta_g}$ and $P_0=\sqrt{-\Delta_{g_0}}$. Let $\mathcal J$ be multiplication by a smooth cutoff $\chi$, followed by identification of an isometric common region containing $\supp\chi$. The Laplacians and volume forms agree on this region. Assume that $\supp d\chi$ is compactly contained in a region where~\eqref{eq:LS-exterior} holds. The cutoff may equal one on an infinite end. Let $f_\lambda$ be spectrally supported in $c_1\le hP\le c_2$, with $0<c_1<c_2$ fixed, and choose $\varphi\in C_c^\infty((0,\infty))$ equal to one on a fixed neighborhood of $[c_1,c_2]$. Put
\[
 u(t)=U(t)f_\lambda,\qquad
 u_{\mathrm{off}}=(1-\varphi(hP_0))\mathcal J u.
\]
For every fixed $m\ge0$ and all sufficiently large integers $N$,
\begin{equation}\label{eq:offband-graph}
 \norm{(1-h^2\Delta_{g_0})^{m/2}u_{\mathrm{off}}}{L^2_{t,x}}
 \le C_{m,N}h^{N+1/2}\norm{f_\lambda}2.
\end{equation}
For a fixed compact smooth curve $\Gamma$ on $Y$, or a complete geodesic $\Gamma$ when $Y$ is Cartan--Hadamard, this implies
\begin{equation}\label{eq:offband-trace}
 \norm{R_\Gamma u_{\mathrm{off}}}{L^3(\R\times\Gamma)}\le C_Ah^A\norm{f_\lambda}2
\end{equation}
for every prescribed $A$.
\end{lemma}
\begin{proof}
Define
\[
 \mathcal C_0=\mathcal J,\qquad \mathcal C_{N+1}=\Delta_{g_0}\mathcal C_N-\mathcal C_N\Delta_g.
\]
For $N\ge1$, $\mathcal C_N$ is a differential operator of order at most $N$, with coefficients supported in $\supp d\chi$. The cancellation of the highest-order terms uses the equality of the two differential expressions. Hence~\eqref{eq:derivative-LS}, with $\lambda=h^{-1}$, gives
\begin{equation}\label{eq:CN-smoothing}
 \norm{\mathcal C_N u}{L^2_{t,x}}\le C_Nh^{1/2-N}\norm{f_\lambda}2.
\end{equation}
For fixed $h$, the spectrally localized solution is bounded in each graph Sobolev norm. Thus $\mathcal J u$ is a graph-Sobolev-valued tempered distribution in time, with temporal Fourier support contained in $\tau\in h^{-2}[c_1^2,c_2^2]$. Since $(\Delta_g+\tau)\widehat u(\tau)=0$, induction gives
\begin{equation}\label{eq:cross-commutator}
 (\Delta_{g_0}+\tau)^N\widehat{\mathcal J u}(\tau)=\widehat{\mathcal C_N u}(\tau).
\end{equation}
For $h^2\tau\in[c_1^2,c_2^2]$, and with $z\ge0$ denoting the spectral variable for $-h^2\Delta_{g_0}$, the scalar multiplier
\[
 (1+z)^{m/2}\frac{1-\varphi(\sqrt z)}{(h^2\tau-z)^N},\qquad z\ge0,
\]
is uniformly bounded when $N\ge m/2$. It is defined to be zero on the fixed neighborhood where its numerator vanishes. Multiplying~\eqref{eq:cross-commutator} by the corresponding separated spectral multiplier therefore gives
\begin{equation}\label{eq:offband-formula}
 \widehat{u_{\mathrm{off}}}(\tau)
 =h^{2N}(1-\varphi(hP_0))(h^2\Delta_{g_0}+h^2\tau)^{-N}
 \widehat{\mathcal C_N u}(\tau).
\end{equation}
The inverse is restricted to the support of $1-\varphi(hP_0)$. To obtain~\eqref{eq:offband-formula}, first insert bounded spatial spectral cutoffs and a smooth temporal multiplier equal to one on the support of $\widehat u$. Pass to the limit in graph-valued distributions. By~\eqref{eq:CN-smoothing} and the preceding scalar bound, the right side belongs to $L^2_{\tau,x}$ after applying $(1-h^2\Delta_{g_0})^{m/2}$. Plancherel proves~\eqref{eq:offband-graph}.

All operators used here are independent of time, so $u_{\mathrm{off}}$ retains temporal bandwidth $O(h^{-2})$. Hilbert-valued Bernstein, the local or global Sobolev trace estimate, and $\|u_{\mathrm{off}}\|_{H^2}\le h^{-2}\|(1-h^2\Delta_{g_0})u_{\mathrm{off}}\|_2$ give
\[
 \norm{R_\Gamma u_{\mathrm{off}}}{L^3_{t,r}}
 \le Ch^{-1/3}\norm{u_{\mathrm{off}}}{L^2_tH^2_x}
 \le C_Nh^{N-11/6}\norm{f_\lambda}2.
\]
Increase $N$ to obtain~\eqref{eq:offband-trace}. In the complete-geodesic case the global trace bound is~\eqref{eq:global-Sobolev-trace}. \end{proof}

\subsection{Compact geodesic segments}
Fix a compact geodesic segment $\gamma$. Choose $\chi_0\in C_c^\infty(M)$ equal to one near $\gamma$ and the trapped projection, with $\supp d\chi_0$ sufficiently far out for~\eqref{eq:LS-exterior}. Let $f_\lambda$ be supported in a fixed annulus $P\asymp\lambda$. Choose $\varphi\in C_c^\infty((0,\infty))$ equal to one on a fixed neighborhood of the $P/\lambda$-spectral support of $f_\lambda$. Put
\[
 u=U(t)f_\lambda,\qquad v=\varphi(P/\lambda)\chi_0 u.
\]
Lemma~\ref{lem:offband}, with $P_0=P$, $\Gamma=\gamma$ and $\mathcal J$ equal to multiplication by $\chi_0$, controls the difference between $R_\gamma u$ and $R_\gamma v$ in $L^3(\R\times\gamma)$ by an arbitrarily large negative power of $\lambda$.

A finite dyadic decomposition extends Theorem~\ref{thm:log} to every fixed frequency annulus. For each fixed integer $m$, the interval lengths $(2^m\lambda)^{-1}\log(2^m\lambda)$ and $\lambda^{-1}\log\lambda$ are comparable for large $\lambda$. Choose $c_*>0$ so that intervals of length $4c_*\lambda^{-1}\log\lambda$ are admissible for all cutoffs below. Define
\[
 v_j(t)=\zeta\left(\frac{\lambda t}{c_*\log\lambda}-j\right)v(t),
 \qquad
 I_j=\left[(j-2)c_*\lambda^{-1}\log\lambda,
              (j+2)c_*\lambda^{-1}\log\lambda\right].
\]
The intervals have bounded overlap and $v_j$ vanishes at their left endpoints. Since $[\Delta_g,\varphi(P/\lambda)]=0$,
\begin{equation}\label{eq:glue-equation}
 (i\partial_t+\Delta_g)v_j=F_j^{(t)}+F_j^{(x)},
\end{equation}
where
\[
 F_j^{(t)}
 =i\frac{\lambda}{c_*\log\lambda}
 \zeta'\left(\frac{\lambda t}{c_*\log\lambda}-j\right)v,
 \qquad
 F_j^{(x)}
 =\zeta\left(\frac{\lambda t}{c_*\log\lambda}-j\right)
 \varphi(P/\lambda)[\Delta_g,\chi_0]u.
\]
Both forcing terms lie in the same enlarged spectral annulus. Let $v_j^{(t)}$ and $v_j^{(x)}$ be their Duhamel solutions with zero data at $\inf I_j$. Then $v_j=v_j^{(t)}+v_j^{(x)}$ on $I_j$.

For the first term, the logarithmic-time homogeneous estimate and Minkowski's inequality imply
\[
 \norm{R_\gamma v_j^{(t)}}{L^3(I_j\times\gamma)}
 \le C\norm{F_j^{(t)}}{L^1_tL^2_x}
 \le C\bigl(\lambda^{-1}\log\lambda\bigr)^{1/2}
 \norm{F_j^{(t)}}{L^2_{t,x}}.
\]
By finite overlap and~\eqref{eq:LS-interior},
\begin{equation}\label{eq:glue-time}
 \sum_j\norm{R_\gamma v_j^{(t)}}{L^3(I_j\times\gamma)}^2
 \le C(\lambda^{-1}\log\lambda)^{-1}\norm v{L^2_{t,x}}^2
 \le C\norm{f_\lambda}2^2.
\end{equation}

For the spatial commutator, choose a real cutoff $\widetilde\chi\in C_c^\infty(M)$ supported in the lossless transition region and equal to one near the coefficients of $[\Delta_g,\chi_0]$. Choose a real $\widetilde\varphi\in C_c^\infty((0,\infty))$ equal to one on $\supp\varphi$. Lemma~\ref{lem:CK-smoothing}, with $B=\varphi(P/\lambda)$ and $\widetilde B=\widetilde\varphi(P/\lambda)$, applies on $I_j$. Theorem~\ref{thm:log} gives $A\lesssim1$, while \eqref{eq:LS-exterior} gives $D\lesssim\lambda^{-1/2}$. Thus
\begin{equation}\label{eq:glue-retarded}
 \left\|R_\gamma\int_{\inf I_j}^{t}U(t-s)\varphi(P/\lambda)\widetilde\chi G(s)\dd s\right\|_{L^3(I_j\times\gamma)}
 \le C\lambda^{-1/2}\norm G{L^2_{s,x}}.
\end{equation}
Apply this with
\[
 G(s)=\zeta\left(\frac{\lambda s}{c_*\log\lambda}-j\right)[\Delta_g,\chi_0]u(s).
\]
Estimate~\eqref{eq:derivative-LS} gives
\begin{equation}\label{eq:glue-space}
 \sum_j\norm{R_\gamma v_j^{(x)}}{L^3(I_j\times\gamma)}^2
 \le C\lambda^{-1}\norm{[\Delta_g,\chi_0]u}{L^2_{t,x}}^2
 \le C\norm{f_\lambda}2^2.
\end{equation}
Equations~\eqref{eq:glue-time}--\eqref{eq:glue-space} give
\[
 \sum_j\norm{R_\gamma v_j}{L^3(I_j\times\gamma)}^2\le C\norm{f_\lambda}2^2.
\]
Hence
\[
 \norm{R_\gamma v}{L^3(\R\times\gamma)}
 \le C\left(\sum_j\norm{R_\gamma v_j}{L^3(I_j\times\gamma)}^3\right)^{1/3}
 \le C\left(\sum_j\norm{R_\gamma v_j}{L^3(I_j\times\gamma)}^2\right)^{1/2}
 \le C\norm{f_\lambda}2.
\]
Finite sums followed by the same tail estimate justify convergence. Adding~\eqref{eq:offband-trace} gives
\begin{equation}\label{eq:compact-high}
 \norm{R_\gamma U(t)f_\lambda}{L^3(\R\times\gamma)}
 \le C_\gamma\norm{f_\lambda}2.
\end{equation}
Sum the frequency bands by~\eqref{eq:temporal-synthesis} to obtain~\eqref{eq:asymp-high} for compact observations.

\subsection{Geodesic tails}
Fix an end $e$ and its auxiliary surface $Y$. Choose $\chi_{\infty,e}$ supported in this component of $M_\infty$ and equal to one sufficiently far out. Its derivatives are supported in the lossless transition region. Let $\mathcal J:L^2(M)\to L^2(Y)$ be the transfer operator defined as in~\eqref{eq:exterior-Duhamel}. The escaping tail extends to a complete geodesic $\Gamma$ on $Y$.

For data $f_\lambda$ spectrally supported in $P\asymp\lambda$, choose $\varphi\in C_c^\infty((0,\infty))$ equal to one on a fixed neighborhood of its frequency annulus, let $u=U(t)f_\lambda$, and set
\[
 v=\varphi(P_0/\lambda)\mathcal J u.
\]
Lemma~\ref{lem:offband}, with $h=\lambda^{-1}$, controls the off-band trace on $\R\times\Gamma$. The remaining term satisfies
\[
 (i\partial_t+\Delta_{g_0})v=\varphi(P_0/\lambda)\mathcal C_1 u,
 \qquad \mathcal C_1=\Delta_{g_0}\mathcal J-\mathcal J\Delta_g.
\]
Its initial datum has $L^2(Y)$ norm at most $C\|f_\lambda\|_2$. The high-frequency part of Proposition~\ref{prop:background} controls its homogeneous evolution.

Let $\widetilde\chi\in C_c^\infty(Y)$ equal one near the coefficients of $\mathcal C_1$ and be supported in the lossless region. Choose a real $\widetilde\varphi\in C_c^\infty((0,\infty))$ equal to one on $\supp\varphi$. Apply Lemma~\ref{lem:CK-smoothing} with $B=\varphi(P_0/\lambda)$, $\widetilde B=\widetilde\varphi(P_0/\lambda)$ and observation $\Gamma$. Use its retarded form on $[0,\infty)$ and its $s>t$ form on $(-\infty,0]$. The high-frequency part of Proposition~\ref{prop:background} bounds the homogeneous factor, and dual exterior smoothing gives $\lambda^{-1/2}$.
From~\eqref{eq:CN-smoothing},
\[
 \norm{\mathcal C_1 u}{L^2_{t,x}}\le C\lambda^{1/2}\norm{f_\lambda}2.
\]
Combining these estimates with Lemma~\ref{lem:offband} gives
\begin{equation}\label{eq:tail-high}
 \norm{R_\Gamma\mathcal J u}{L^3(\R\times\Gamma)}\le C\norm{f_\lambda}2.
\end{equation}
Restriction to the identified ray can only decrease the norm.

A complete nontrapped geodesic has two exterior tails and a compact middle parameter interval. Choose $\chi_{\infty,e}=1$ on a sufficiently late part of each tail. Estimate~\eqref{eq:tail-high} controls these two parts. Cover the middle interval by finitely many short segments and apply~\eqref{eq:compact-high}. This proves the single-band bound on the complete geodesic. Summation by~\eqref{eq:temporal-synthesis} proves~\eqref{eq:asymp-high}.

\subsection{Low frequencies}
Assume that the threshold on $M$ is regular. Let $\lambda_0$ be the high-frequency cutoff in~\eqref{eq:asymp-high}, and choose $\Theta\in C_c^\infty([0,\infty))$ equal to one on $[0,2\lambda_0]$. Put
\[
 u_{\mathrm{lo}}(t)=U(t)\Theta(P)\Pc f.
\]
For a compact observed geodesic $\gamma$,~\eqref{eq:LS-low} and local elliptic regularity give, with a compact spatial cutoff $\chi_{\mathrm{lo}}=1$ near $\gamma$,
\[
 \norm{\chi_{\mathrm{lo}} u_{\mathrm{lo}}}{L^2_tH^2_x}\le C\norm f2.
\]
The temporal spectrum remains in a fixed bounded interval. Hilbert-valued temporal Bernstein and the local Sobolev trace theorem imply
\begin{equation}\label{eq:compact-low}
 \norm{R_\gamma u_{\mathrm{lo}}}{L^3(\R\times\gamma)}\le C\norm f2.
\end{equation}

For an infinite tail, the transferred solution satisfies
\[
 (i\partial_t+\Delta_{g_0})\mathcal J u_{\mathrm{lo}}
 =\mathcal C_1 u_{\mathrm{lo}}.
\]
Its initial datum has $L^2(Y)$ norm at most $C\|f\|_2$. Since the chosen background has no eigenvalues,~\eqref{eq:background-full} controls the homogeneous term. For the source,~\eqref{eq:LS-low} and local elliptic regularity give
\[
 \|\mathcal C_1u_{\mathrm{lo}}\|_{L^2_tH^2(Y)}\le C\|f\|_2.
\]
Here the original spectral support is bounded, and $\mathcal C_1$ has compactly supported coefficients. Choose a real $\chi_{\mathrm{lo}}\in C_c^\infty(Y)$ equal to one near these coefficients. For any fixed $L\ge1$ apply Lemma~\ref{lem:CK-smoothing} on each half-line with $B=\1_{[0,L]}(P_0)$ and $\widetilde B=\Id$. The background restriction estimate \eqref{eq:background-full} and background smoothing \eqref{eq:background-full-smoothing} give, uniformly in $L$,
\[
 \left\|R_\Gamma\int_{\{s\ge0:s<t\}}U_0(t-s)\1_{[0,L]}(P_0)\chi_{\mathrm{lo}}F(s)\dd s\right\|_{L^3(\R_+\times\Gamma)}
 \le C\norm F{L^2(\R_+\times Y)}.
\]
The analogous estimate on negative times follows by time reversal.

Take $F=\mathcal C_1u_{\mathrm{lo}}$, so $\chi_{\mathrm{lo}}F=F$. For every fixed $T>0$, unitarity on $H^2(Y)$ gives
\[
 \sup_{0\le t\le T}
 \left\|\int_0^t U_0(t-s)(1-\1_{[0,L]}(P_0))F(s)\dd s\right\|_{H^2}
 \le T^{1/2}\|(1-\1_{[0,L]}(P_0))F\|_{L^2([0,T];H^2)}
 \longrightarrow0.
\]
The last limit follows from the spectral theorem and dominated convergence. The same argument applies on $[-T,0]$. Thus the projected Duhamel integrals converge to the transferred Duhamel solution in $C([-T,T];H^2(Y))$, and~\eqref{eq:global-Sobolev-trace} gives convergence of their traces in $L^3([-T,T]\times\Gamma)$. The preceding retarded bound is independent of $L$ and $T$. Letting $L\to\infty$ and then $T\to\infty$ proves~\eqref{eq:tail-high} for $u_{\mathrm{lo}}$, together with the homogeneous bound.

The finite middle interval is controlled by~\eqref{eq:compact-low}. Decompose
\[
 \Pc f=\Theta(P)\Pc f+(1-\Theta(P))\Pc f.
\]
The first term is controlled by the preceding argument and the second by~\eqref{eq:asymp-high}. This proves~\eqref{eq:asymp-full} and Theorem~\ref{thm:asymp}.\qed

\section{Sharpness and the threshold}\label{sec:sharp}
\subsection{Spectral windows}\label{sec:sharp-static}
We present two examples, following the Fourier representations in~\cite[Section~3]{AGL}, which prove the sharpness of the two branches in Theorem~\ref{thm:AH-static} on the hyperbolic plane. For these examples, write
\[
 Q_{\HH}=\sqrt{-\Delta_{\HH}-1/4}.
\]
The examples use the spectral parameter of $Q_{\HH}$. For an exact window of $\sqrt{-\Delta_{\HH}}$, map its endpoints by $\nu\mapsto\sqrt{\nu^2-1/4}$. For $\nu\ge1$, the transformed width is comparable to the original width uniformly for arbitrarily small windows.

\begin{example}[Knapp example]\label{ex:Knapp}
Use the boundary Fourier normalization of~\cite{AGL}, for which
\[
 \norm f{L^2(\HH)}^2\sim
 \int_0^\infty\int_\R \kappa^2|\widetilde f(\kappa,\xi)|^2\dd\xi\dd \kappa.
\]
Let $I_{\lambda,\eta}\subset[\lambda,\lambda+\eta]$ have length comparable to $\eta$, and choose a nonzero smooth nonnegative $b$ supported in $(-1,1)$. Put $y_\lambda=C\sqrt\lambda$ for a sufficiently large fixed $C$ and define
\begin{equation}\label{eq:Knapp-data}
 \widetilde f(\kappa,\xi)=
 \1_{I_{\lambda,\eta}}(\kappa)b(\xi)e^{i\kappa\log y_\lambda}.
\end{equation}
The additional factor has modulus one. Hence
\begin{equation}\label{eq:Knapp-input}
 \norm f{L^2(\HH)}\sim\lambda\eta^{1/2}.
\end{equation}
In the upper half-plane coordinates $x+iy$, Fourier inversion has kernel
\[
 c\,\frac{\Gamma(\frac12+i\kappa)}{\Gamma(1+i\kappa)}
 \left(\frac{y}{(x-\xi)^2+y^2}\right)^{1/2}
 e^{i\kappa\log(y/((x-\xi)^2+y^2))}\,\kappa^2.
\]
Consider the fixed vertical geodesic $x=1$ and the interval
\[
 e^{-\epsilon}\le y/y_\lambda\le e^\epsilon.
\]
This has fixed positive hyperbolic length. The combined phase, including~\eqref{eq:Knapp-data}, is
\[
 \kappa\left[-\log(y/y_\lambda)-\log\left(1+\frac{(1-\xi)^2}{y^2}\right)\right].
\]
Remove the common central-frequency factor in the first term. Its remaining variation is at most $C\eta\epsilon$. The second term has absolute size at most $C_1/C^2$, since $\kappa\sim\lambda$ and $y\sim C\sqrt\lambda$. Finally, Stirling's formula gives
\[
 \frac{\Gamma(\frac12+i\kappa)}{\Gamma(1+i\kappa)}
 =e^{-i\pi/4}\kappa^{-1/2}(1+O(\kappa^{-1})).
\]
Choose $C$ large and then $\epsilon$ small. After multiplication by a common unit complex number, the real part of the entire integrand is at least a fixed positive proportion of its absolute value. Thus
\begin{equation}\label{eq:Knapp-output}
 |f(1+iy)|\gtrsim\lambda^{3/2}\eta y^{-1/2}
 \gtrsim\lambda^{5/4}\eta
\end{equation}
throughout this interval. Consequently, for every fixed $2\le q\le4$,
\[
 \norm{R_\gamma f}{L^q(\gamma)}\gtrsim\lambda^{5/4}\eta,
 \qquad
 \frac{\norm{R_\gamma f}q}{\norm f2}
 \gtrsim\lambda^{1/4}\eta^{1/2}.
\]
The observation interval moves up the complete geodesic as $\lambda$ increases. The phase factor in~\eqref{eq:Knapp-data} makes the lower bound uniform for $0<\eta\le1$.
\end{example}

\begin{example}[Spherical example]\label{ex:spherical}
Choose a radial function whose spherical Fourier transform is the indicator of an interval $I_{\lambda,\eta}\subset[\lambda,\lambda+\eta]$ of comparable length, with the usual even continuation if the transform is defined on the whole real line. Plancherel gives
\begin{equation}\label{eq:spherical-input}
 \norm f2^2\sim\int_{I_{\lambda,\eta}}
 \left|\frac{\Gamma(\frac12+i\kappa)}{\Gamma(i\kappa)}\right|^2\dd \kappa
 \sim\lambda\eta.
\end{equation}
The spherical function is
\[
 \varphi_\kappa(r)=\frac{\sqrt2}{\pi}
 \int_0^r\cos(\kappa s)(\cosh r-\cosh s)^{-1/2}\dd s.
\]
For $0\le r\le c/\lambda$ with $c$ fixed and small, it is bounded below by a positive constant uniformly for $\kappa\in I_{\lambda,\eta}$. Spherical inversion therefore gives $f(r)\gtrsim\lambda\eta$ on this ball. On any diameter geodesic,
\[
 \norm{R_\gamma f}q\gtrsim\lambda^{1-1/q}\eta,
 \qquad
 \frac{\norm{R_\gamma f}q}{\norm f2}
 \gtrsim\lambda^{1/2-1/q}\eta^{1/2}.
\]
This proves sharpness for $4\le q\le\infty$. Together with Example~\ref{ex:Knapp}, the examples give both powers in~\eqref{eq:main-AH}.
\end{example}

We finally give an example illustrating the bounded-geometry hypothesis.
\begin{example}[A hyperbolic surface with a cusp]\label{ex:cusp}
Let $X=\HH/\langle x\mapsto x+\alpha\rangle$, $\alpha>0$. The infimum of its injectivity radius is zero, so $X$ does not have bounded geometry. The nonnegative hyperbolic Laplacian is
\[
 -\Delta_\HH=-y^2(\partial_x^2+\partial_y^2).
\]
In particular
\[
 (-\Delta_\HH-1/4)y^{1/2-i\xi}=\xi^2y^{1/2-i\xi}.
\]
These functions are independent of $x$ and descend to the cylinder. The constant angular mode, under $r=\log y$ and multiplication by $y^{-1/2}$, identifies $-\Delta_X-1/4$ with $-\partial_r^2$ on $L^2(\R,\dd r)$.

Let $\phi(s)=a(s)\1_{[0,\infty)}(s)$, where $a\in C_c^\infty((-1/10,1/10))$ and $a(0)=1$. Put
\[
 \varsigma_\lambda(x,y)=\frac1{\sqrt{2\pi}}
 \int_\R \phi(\eta^{-1}(\xi-\lambda))y^{1/2-i\xi}\dd\xi.
\]
Its $\sqrt{-\Delta_X-1/4}$ spectrum is contained in $[\lambda,\lambda+\eta]$. Fourier inversion and Plancherel give
\[
 \varsigma_\lambda(x,y)=c\eta y^{1/2-i\lambda}\widehat\phi(-\eta\log y),
 \qquad
 \norm{\varsigma_\lambda}{L^2(X)}\sim_\alpha\eta^{1/2}\norm\phi2.
\]
For the vertical geodesic $\gamma(y)=iy$, whose arclength measure is $\dd y/y$,
\[
 \norm{R_\gamma\varsigma_\lambda}q^q
 =C\eta^q\int_\R e^{q r/2}|\widehat\phi(-\eta r)|^q\dd r.
\]
Integration by parts on the positive half-line gives
$\widehat\phi(\tau)=-(i\tau)^{-1}+O(|\tau|^{-2})$ as $|\tau|\to\infty$. Thus $\varsigma_\lambda\in L^2(X)$, but its restriction is not in $L^q(\gamma)$ for every finite $q>2$, and the trace is also unbounded. The estimate in Theorem~\ref{thm:AH-static} therefore fails on this cusp surface.
\end{example}

\subsection{Spacetime concentration}
\begin{proposition}\label{prop:sharp}
Suppose that $\beta\not\equiv0$ and that the observed geodesic $\gamma$ contains a nonempty open segment. There exist $L^2$-normalized data $f_\lambda$ and constants $c_0,c_1>0$ such that
\begin{equation}\label{eq:sharp-lower}
 \norm{R_\gamma U(t)\beta(P/\lambda)f_\lambda}
 {L^3([0,c_0\lambda^{-2}]\times\gamma)}\ge c_1
\end{equation}
for all sufficiently large $\lambda$. Consequently, for nonzero frequency cutoffs the $O(1)$ high-frequency bounds in Theorems~\ref{thm:log} and~\ref{thm:asymp}, and in Proposition~\ref{prop:background}, cannot be replaced by $o(1)$ on their respective observation intervals.
\end{proposition}
\begin{proof}
Choose normal coordinates $x$ at an interior point of the observed segment, with $g(0)$ Euclidean and the tangent geodesic given by $x_2=0$. Let $f_0$ be a Schwartz function whose Fourier transform is supported in a small annulus where $\beta$ is nonzero, and arrange that $[\beta(\sqrt{-\Delta_{\R^2}})f_0](0)\ne0$. Let $\chi_0$ be a fixed coordinate cutoff equal to one near zero, and set
\[
 f_\lambda(x)=c_\lambda\lambda\chi_0(x)f_0(\lambda x),\qquad \norm{f_\lambda}{L^2(M)}=1.
\]
The normalization constants have a finite nonzero limit. For bounded $s$, the functions $z\mapsto e^{-isz}\beta(\sqrt z)$ form a uniformly bounded family in $C_c^\infty((0,\infty))$. Applying these multipliers to $-\lambda^{-2}\Delta_g$ gives $U(\lambda^{-2}s)\beta(P/\lambda)$. Their local semiclassical principal symbols are
\[
 e^{-is|\xi|_{g(x)}^2}\beta(|\xi|_{g(x)}).
\]
Rescaling the coordinate kernel by $x=\lambda^{-1}X$ and the input variable by $\lambda^{-1}$, the local functional calculus and the Schwartz decay of $f_0$ therefore give
\[
 \lambda^{-1}[U(\lambda^{-2}s)\beta(P/\lambda)f_\lambda](\lambda^{-1}X)
 \longrightarrow c\,e^{is\Delta_{\R^2}}\beta(\sqrt{-\Delta_{\R^2}})f_0(X)
\]
uniformly on bounded rescaled space and time sets. The right side is nonzero on a small rectangle in $(s,X_1)$ at $X_2=0$. Thus $|R_\gamma U(t)\beta(P/\lambda)f_\lambda|\ge c\lambda$ on a set of time length comparable to $\lambda^{-2}$ and curve length comparable to $\lambda^{-1}$. Its $L^3$ norm is bounded below by
\[
 c\lambda(\lambda^{-2}\lambda^{-1})^{1/3}=c,
\]
which proves~\eqref{eq:sharp-lower}. The same local functional calculus gives
\[
 \liminf_{\lambda\to\infty}\norm{\beta(P/\lambda)f_\lambda}2>0.
\]
Thus, after normalizing $\beta(P/\lambda)f_\lambda$, the high-pass projector in Theorem~\ref{thm:asymp} acts as the identity for all sufficiently large $\lambda$, and the same lower bound applies. The small time interval lies inside every logarithmic interval in Theorem~\ref{thm:log} for large $\lambda$, and inside the all-time intervals in the other assertions. At sufficiently high frequency the spectral cutoffs in the asymptotically hyperbolic statements lie above the discrete spectrum.
\end{proof}

\subsection{A threshold obstruction}
\begin{proposition}[An observed threshold term]\label{prop:threshold-obstruction}
Let $\gamma:I\to M$ be a compact observed geodesic segment on an asymptotically hyperbolic surface, and let $Q=\sqrt{-\Delta_g-1/4}$ on the absolutely continuous spectral subspace. Suppose that its observed spectral density has the expansion
\begin{equation}\label{eq:threshold-leading}
 R_\gamma\frac{dE_Q(\kappa)}{d\kappa}R_\gamma^*
 =\mathcal A_I+O(\kappa^2),\qquad \kappa\downarrow0,
\end{equation}
in smooth kernels on $I\times I$, where $\mathcal A_I\ne0$. Then, for all sufficiently small $\varepsilon>0$, there are $L^2$-normalized data $f_\varepsilon$ in the absolutely continuous spectral subspace, supported in $\varepsilon\le Q\le2\varepsilon$, such that
\begin{equation}\label{eq:threshold-lower}
 \norm{R_\gamma U(t)f_\varepsilon}{L^3(\R\times\gamma)}
 \ge c\varepsilon^{-1/6}.
\end{equation}
In particular, an unweighted all-time $L^2\to L^3$ trace estimate fails. Any bound with a right-hand multiplier comparable to $Q^{-\alpha}$ at the threshold must have $\alpha\ge1/6$.
\end{proposition}
\begin{proof}
Positivity of the spectral measure and~\eqref{eq:threshold-leading} imply $\mathcal A_I\ge0$. Choose $b\in C_c^\infty(I)$ with $\ip{\mathcal A_Ib}{b}>0$, and put
\[
 f_\varepsilon=
 \frac{(R_\gamma\1_{[\varepsilon,2\varepsilon]}(Q)\Pc)^*b}
 {\|(R_\gamma\1_{[\varepsilon,2\varepsilon]}(Q)\Pc)^*b\|_2}.
\]
Finite-band spectral kernels make these well-defined $L^2$ data. By~\eqref{eq:threshold-leading}, the squared denominator is
\[
 \int_\varepsilon^{2\varepsilon}
 \left\langle R_\gamma\frac{dE_Q(\kappa)}{d\kappa}R_\gamma^*b,b\right\rangle\dd \kappa
 =\ip{\mathcal A_Ib}{b}\,\varepsilon+O(\varepsilon^3).
\]
For $0\le t\le c_0\varepsilon^{-2}$, with fixed $c_0>0$ sufficiently small, the real part of $e^{-it\kappa^2}$ is uniformly positive throughout this band. Positivity of the scalar spectral density, after removing $e^{-it/4}$, therefore gives
\[
 \left|\left\langle e^{it/4}R_\gamma U(t)f_\varepsilon,b\right\rangle\right|
 \ge c\varepsilon^{1/2}.
\]
Spatial H\"older and integration over that time interval imply
\[
 \norm{R_\gamma U(t)f_\varepsilon}{L^3(\R\times\gamma)}
 \ge c\varepsilon^{1/2}(\varepsilon^{-2})^{1/3}
 =c\varepsilon^{-1/6}.
\]
On the same band, $\|Q^{-\alpha}f_\varepsilon\|_2\asymp\varepsilon^{-\alpha}$. Comparison proves the final assertion.
\end{proof}

A first-order pole of the continued shifted resolvent can produce~\eqref{eq:threshold-leading}. Stone's formula multiplies its odd part by $\kappa/(\pi i)$, so a nonzero observed $\kappa^{-1}$ coefficient gives a constant term in the spectral density. See~\cite[Section~1.5 and (1.3)]{CH}. The lower bound also holds on any complete geodesic containing the observed compact segment.

\section*{Use of generative AI}
The author used ChatGPT (OpenAI) to assist with literature searches, exploratory proof development, checks of intermediate calculations and references, and drafting and editing parts of the manuscript. The author takes responsibility for the mathematical arguments and the final text.

\bibliographystyle{amsplain}
\bibliography{main}

\providecommand{\bysame}{\leavevmode\hbox to3em{\hrulefill}\thinspace}
\providecommand{\MR}{\relax\ifhmode\unskip\space\fi MR }
\providecommand{\MRhref}[2]{%
  \href{http://www.ams.org/mathscinet-getitem?mr=#1}{#2}
}
\providecommand{\href}[2]{#2}
\begin{thebibliography}{10}

\bibitem{AGL}
Jean-Philippe Anker, Pierre Germain, and Tristan L\'eger, \emph{Spectral
  projectors on hyperbolic surfaces}, 2023,
  \url{https://arxiv.org/abs/2306.12827}.

\bibitem{Blair}
M.~D. Blair, \emph{On logarithmic improvements of critical geodesic restriction
  bounds in the presence of nonpositive curvature}, Israel J. Math.
  \textbf{224} (2018), 407--436, \url{https://arxiv.org/abs/1607.03174}.

\bibitem{Borthwick}
David Borthwick, \emph{Spectral theory of infinite-area hyperbolic surfaces}, 2
  ed., Progress in Mathematics, vol. 318, Birkh\"auser, 2016.

\bibitem{BR}
Jean Bourgain and Ze\'ev Rudnick, \emph{Restriction of toral eigenfunctions to
  hypersurfaces}, C. R. Math. Acad. Sci. Paris \textbf{347} (2009), 1249--1253.

\bibitem{BGT-Strichartz}
N.~Burq, P.~G\'erard, and N.~Tzvetkov, \emph{Strichartz inequalities and the
  nonlinear {Schr\"odinger} equation on compact manifolds}, Amer. J. Math.
  \textbf{126} (2004), 569--605.

\bibitem{BGT-restriction}
\bysame, \emph{Restrictions of the {Laplace--Beltrami} eigenfunctions to
  submanifolds}, Duke Math. J. \textbf{138} (2007), 445--486,
  \url{https://arxiv.org/abs/math/0506394}.

\bibitem{BGH}
N.~Burq, C.~Guillarmou, and A.~Hassell, \emph{Strichartz estimates without loss
  on manifolds with hyperbolic trapped geodesics}, Geom. Funct. Anal.
  \textbf{20} (2010), 627--656, \url{https://arxiv.org/abs/0907.3545}.

\bibitem{ChenIII}
X.~Chen, \emph{Resolvent and spectral measure on non-trapping asymptotically
  hyperbolic manifolds {III}: Global-in-time {Strichartz} estimates without
  loss}, Ann. Inst. H. Poincar\'e C Anal. Non Lin\'eaire \textbf{35} (2018),
  803--829, \url{https://doi.org/10.1016/j.anihpc.2017.08.003}.

\bibitem{CHI}
X.~Chen and A.~Hassell, \emph{Resolvent and spectral measure on non-trapping
  asymptotically hyperbolic manifolds {I}: Resolvent construction at high
  energy}, Comm. Partial Differential Equations \textbf{41} (2016), no.~3,
  515--578, \url{https://doi.org/10.1080/03605302.2015.1116561}.

\bibitem{CH}
\bysame, \emph{Resolvent and spectral measure on non-trapping asymptotically
  hyperbolic manifolds {II}: Spectral measure, restriction theorem, spectral
  multipliers}, Ann. Inst. Fourier (Grenoble) \textbf{68} (2018), no.~3,
  1011--1075, \url{https://doi.org/10.5802/aif.3183}.

\bibitem{Chen}
Xuehua Chen, \emph{An improvement on eigenfunction restriction estimates for
  compact boundaryless {Riemannian} manifolds with nonpositive sectional
  curvature}, Trans. Amer. Math. Soc. \textbf{367} (2015), 4019--4039.

\bibitem{ChenSogge}
Xuehua Chen and Christopher~D. Sogge, \emph{A few endpoint geodesic restriction
  estimates for eigenfunctions}, Comm. Math. Phys. \textbf{329} (2014),
  435--459.

\bibitem{CK}
M.~Christ and A.~Kiselev, \emph{Maximal functions associated to filtrations},
  J. Funct. Anal. \textbf{179} (2001), 409--425,
  \url{https://sites.math.duke.edu/~kiselev/maxim.pdf}.

\bibitem{Eldering}
Jaap Eldering, \emph{Normally hyperbolic invariant manifolds: The noncompact
  case}, Atlantis Studies in Dynamical Systems, vol.~2, Atlantis Press, 2013.

\bibitem{Hou}
Jiaqi Hou, \emph{Restrictions of {Maass} forms on {$\mathrm{SL}(2,\mathbb{C})$}
  to hyperbolic surfaces and geodesic tubes}, 2024,
  \url{https://arxiv.org/abs/2410.17164}.

\bibitem{Hu}
Rui Hu, \emph{{$L^p$} norm estimates of eigenfunctions restricted to
  submanifolds}, Forum Math. \textbf{21} (2009), 1021--1052.

\bibitem{HS}
X.~Huang and C.~D. Sogge, \emph{Strichartz estimates for the {Schr\"odinger}
  equation on compact manifolds with nonpositive sectional curvature}, J.
  Spectr. Theory \textbf{16} (2026), no.~3, 777--806,
  \url{https://doi.org/10.4171/JST/551}.

\bibitem{HSTZ}
X.~Huang, C.~D. Sogge, Z.~Tao, and Z.~Zhang, \emph{Lossless {Strichartz} and
  spectral projection estimates on unbounded manifolds}, Geom. Funct. Anal.
  \textbf{36} (2026), no.~1, 59--151,
  \url{https://doi.org/10.1007/s00039-026-00732-6}.

\bibitem{HWZ}
Xiaoqi Huang, Xing Wang, and Cheng Zhang, \emph{Restriction of {Schr\"odinger}
  eigenfunctions to submanifolds}, Comm. Math. Phys. \textbf{407} (2026), article~67,
  \url{https://doi.org/10.1007/s00220-026-05576-7}.

\bibitem{Keeler}
B.~Keeler, \emph{A logarithmic improvement in the two-point {Weyl} law for
  manifolds without conjugate points}, Ann. Inst. Fourier (Grenoble)
  \textbf{74} (2024), 719--762, \url{https://arxiv.org/abs/1905.05136}.

\bibitem{Marshall}
Simon Marshall, \emph{Geodesic restrictions of arithmetic eigenfunctions}, Duke
  Math. J. \textbf{165} (2016), no.~3, 463--508.

\bibitem{Nicola}
F.~Nicola, \emph{Dynamical restriction for {Schr\"odinger} equations}, J.
  Funct. Anal. \textbf{291} (2026), 111467,
  \url{https://doi.org/10.1016/j.jfa.2026.111467}.

\bibitem{Reznikov}
Andre Reznikov, \emph{Norms of geodesic restrictions for eigenfunctions on
  hyperbolic surfaces and representation theory}, 2004,
  \url{https://arxiv.org/abs/math/0403437}.

\bibitem{SZ}
C.~D. Sogge and S.~Zelditch, \emph{On eigenfunction restriction estimates and
  {$L^4$}-bounds for compact surfaces with nonpositive curvature}, Advances in
  Analysis: The Legacy of Elias M. Stein, Princeton Math. Ser., vol.~50,
  Princeton Univ. Press, 2014, \url{https://arxiv.org/abs/1108.2726},
  pp.~447--461.

\bibitem{Tacy}
M.~Tacy, \emph{Semiclassical {$L^p$} estimates of quasimodes on submanifolds},
  Comm. Partial Differential Equations \textbf{35} (2010), 1538--1562,
  \url{https://arxiv.org/abs/0905.2240v2}.

\bibitem{XiZhang}
Y.~Xi and C.~Zhang, \emph{Improved critical eigenfunction restriction estimates
  on {Riemannian} surfaces with nonpositive curvature}, Comm. Math. Phys.
  \textbf{350} (2017), 1299--1325, \url{https://arxiv.org/abs/1603.01601}.

\bibitem{Xi}
Yakun Xi, \emph{Kakeya--nikodym problems and geodesic restriction estimates for
  eigenfunctions}, Ph.D. thesis, Johns Hopkins University, 2017.

\bibitem{CZhang}
Cheng Zhang, \emph{Improved critical eigenfunction restriction estimates on
  {Riemannian} manifolds with constant negative curvature}, J. Funct. Anal.
  \textbf{272} (2017), 4642--4670.

\end{thebibliography}
\end{document}